\documentclass[11pt,a4paper,twoside]{article}

\usepackage[T1]{fontenc}
\usepackage[utf8]{inputenc}
\usepackage{times}

\usepackage[leqno]{amsmath} 
\usepackage{amsthm,amsfonts} 
\usepackage{bm} 

\input xy  \xyoption{all}

\usepackage{geometry}
\usepackage[usenames]{color} 

\usepackage{graphicx} 

\usepackage{enumerate} 
\usepackage{hyperref} 
\usepackage{framed}

\DeclareMathAlphabet{\mathpzc}{OT1}{pzc}{m}{it}

\newcommand{\that}{\ |\ }
\newcommand{\R}{\mathbb{R}}

\newcommand{\wt}[1]{\bm{#1}}

\newcommand{\T}{\mathrm{T}}
\newcommand{\lra}{\longrightarrow}
\newcommand{\ra}{\rightarrow}
\newcommand{\eps}{\varepsilon}

\newcommand{\dd}{\mathrm{d}}

\newcommand{\pa}{\partial}
\newcommand{\image}{\operatorname{Im}}
\def\<#1>{\big\langle #1\big\rangle}
\def\(#1){\left( #1\right)}

\newcommand{\vv}{\bm{v}}
\newcommand{\covector}{\bm{\psi}}
\newcommand{\zero}{\bm{0}}
\newcommand{\G}[1]{\mathcal{Q}_{[#1]}}
\newcommand{\cost}[2]{Q_{#1}[#2]} 
\newcommand{\blincost}[3]{B_{#1}[#2,#3]} 
\newcommand{\costmin}[2]{\widehat{Q}_{#1}(#2)} 
\newcommand{\RR}{\bm{\mathcal{R}}_{(\Lambda_1)}} 
\newcommand{\qq}[1]{\bm{q}^{(#1)}} 
\newcommand{\q}[1]{q^{(#1)}}  
\newcommand{\dotq}[1]{{\dot{q}}^{(#1)}} 
\newcommand{\lin}[1]{\Phi^{(#1)}} 
\newcommand{\dotlin}[1]{{\dot\Phi}^{(#1)}}
\newcommand{\Energy}[1]{E^{[#1]}} 
\newcommand{\Der}{\operatorname{D}} 
\newcommand{\cc}[1]{c^{(#1)}} 
\newcommand{\bbb}[1]{\bm{b}^{(#1)}} 
\newcommand{\bb}[1]{b^{(#1)}} 
\newcommand{\LL}[2]{L^\infty([#1,#2],\R^k)} 
\newcommand{\Ll}[2]{L^2([#1,#2],\R^k)} 
\newcommand{\End}[2]{\mathrm{End}^{[#1]}_{#2}} 
\newcommand{\END}[2]{{\mathbf{End}}^{[#1]}_{#2}} 
\newcommand{\Eng}{\mathcal{E}} 
\newcommand{\D}{\mathcal{D}} 
\newcommand{\U}[1]{\mathcal{U}_{#1}} 
\newcommand{\Sol}[2]{\operatorname{Sol}_{#1}{#2}} 
\newcommand{\FF}{\bm\Psi_2}
\newcommand{\PP}{\bm\Phi_2}

\newcommand{\Pp}{\Phi_2}
\newcommand{\dotPp}{\dot{\Phi}_2}
\newcommand{\Y}[2]{Y_{#1}^{(#2)}} 

\newcommand{\du}{\Delta u} 
\newcommand{\dv}{\Delta v}

\numberwithin{equation}{section} 

\theoremstyle{plain} 
\newtheorem{theorem}{Theorem}[section]
\newtheorem{proposition}[theorem]{Proposition}
\newtheorem{lemma}[theorem]{Lemma}

\newtheorem{corollary}[theorem]{Corollary}

\theoremstyle{definition}
\newtheorem{definition}[theorem]{Definition}

\theoremstyle{remark}
\newtheorem{remark}[theorem]{Remark}

\title{New second-order optimality conditions in sub-Riemannian Geometry}
\author{Michał Jóźwikowski\footnote{Part of this research was conducted during the employment of MJ at the University of Fribourg, finaced by the ERC Starting Grant \emph{Geometry of Metric Groups}, grant agreement 713998 GeoMeG. } \\[0.5cm]
Faculty of Mathematics, Informatics and Mechanics\\
University of Warsaw\\[2ex]
 Department of Mathematics\\
University of Fribourg
}
\date{\today}

\begin{document}
\maketitle
\begin{abstract}
We study the geometry of the second-order expansion of the extended end-point map for the sub-Riemannian geodesic problem. Translating the geometric reality into equations we derive new second-order necessary optimality conditions in sub-Riemannian Geometry. In particular, we find an ODE for velocity of an abnormal sub-Riemannian geodesics. 
It allows to divide abnormal minimizers into two classes, which we propose to call 2-normal and 2-abnormal extremals. In the 2-normal case the above ODE completely determines the velocity of a curve, while in the 2-abnormal case the velocity is undetermined at some, or at all points. With some enhancement of the presented results it should be possible to prove the regularity of all 2-normal extremals (the 2-abnormal case seems to require study of higher-order conditions) thus making a step towards solving the problem of smoothness of sub-Riemannian abnormal geodesics.

As a by-product we present a new derivation of Goh conditions.  We also prove that the assumptions weaker than these used in [Boarotto, Monti, Palmurella, 2020] to derive third-order Goh conditions, imply piece-wise-$C^2$ regularity of an abnormal extremal.
\end{abstract}
\begin{flushright}
\emph{Martynie i Zuzi}
\end{flushright}

\section{Introduction}

\paragraph{The problem of smoothness of abnormal geodesics.}
The question weather all minimizing sub-Riemannian geodesics are smooth dates back to the famous example of Montgomery \cite{Montgomery_example_1994} of a minimizing abnormal geodesics. It is commonly agreed \cite{Agrachev_open_problems_2014, Montgomery_2006} to be one of the main open problem in sub-Riemannian geometry. Despite 30 years of efforts, there are not many general results. Among the latter we must mention the derivation of the second-order necessary conditions for optimality \cite{Agrachev_Sarychev_1996} (known as \emph{Goh conditions}) and a technical masterpiece of Hakavouri and Le Donne \cite{no_corners_2016} (se also \cite{Hakavuori_Donne_2018}) which ruled out corner-like singularities (some previous results in this direction are \cite{Leonardi_Monti_2008,Monti_regularity_results_2014}). We are also aware of a very recent attempts to derive third-order (and higher) necessary conditions of optimality \cite{Monti_third_order_2020, higher_Goh_2022}. We refer to \cite{Monti_regularity_2014} for a discussion of the problem and its history. 

\paragraph{Basic ideas.}
In this paper we propose the following approach to the problem. Consider a sub-Riemannian manifold $(M,\D,\<\cdot,\cdot>)$ and the problem of minimizing the energy by a $\D$-horizontal curve $q(t)$, with $t\in[t_0,t_1]$, joining given two points $q_0=q(t_0)$ and $q_1=q(t_1)$ on $M$. We may (locally) reformulate this problem as an optimal control problem in a standard way -- see page \pageref{eqn:trajectory_u}. Now let
$$\END{t_0,t}{}:\LL{t_0}{t}\lra M\times\R\ ;\qquad \END{t_0,t}{}[u]=(q(t),\Energy{t_0,t}[u]) $$
be the extended end-point map at time $t\in[t_0,t_1]$ related with this optimal control problem.  Here  $q(t)$ is the sub-Riemannian trajectory corresponding to the control $u\in\LL{t_0}{t_1}$ and $\Energy{t_0,t}[u]$ is the energy of this trajectory at time $t$. 

In the next step, for a  distinguished trajectory of the system  $q(t)\in M$ corresponding to a given control $u\in \LL{t_0}{t}$, we study curves $s\mapsto \END{t_0,t}{}[u+s\cdot \du]$ built for all possible controls $\du\in \LL{t_0}{t}$. More precisely, we are interested in 2-jets of such curves at $s=0$, i.e. 
$$\END{t_0,t}{}[u+s\cdot \du]\overset{\text{loc.}}=:\END{t_0,t}{}[u]+ s\cdot \bbb{1}(t,\du)+ {s^2}\cdot \bbb{2}(t,\du)+o(s^2)\ . $$
One may think of the collection of pairs $(\bbb{1}(t,\du),\bbb{2}(t,\du))$ (for all possible $\du$'s) as of quadratic approximations of the image of the end-point map around $\END{t_0,t}{}[u]$.\smallskip

It turns out that, for a given $\du\in\LL{t_0}{t_1}$, curves $\bbb{1}(t,\du)$ and $\bbb{2}(t,\du)$ satisfy  quite a complicated system of ODEs \eqref{eqn:b1}-\eqref{eqn:b2}. It is convenient to treat the latter as a control system (build on top of a fixed trajectory $q(t)$ of the initial sub-Riemannian control system) with $\du$ playing the role of the control.  \medskip

The key point now is to relate the properties of this new controls system with optimality of a sub-Riemannian trajectory $q(t)$. To do this we rely on two results:
\begin{itemize}
    \item Firstly, the classical theory of Agrachev and Sarychev \cite{Agrachev_Sarychev_1996} gives information about the space of curves $\bbb{2}(t,\du)$ for these controls $\du\in\LL{t_0}{t_1}$ which satisfy $\bbb{1}(t,\du)=\zero$.
    \item Secondly, thanks to the results of \cite{MJ_BS_adapted}, we may change coordinates transforming the pairs $(\bbb{1}(t,\du),\bbb{2}(t,\du))$ into pairs $(\qq{1}(t,\du),\qq{2}(t,\du))\in \R^{n+1}\times\R^{n+1}$, for which the new control system we want to analyze takes a much simpler form \eqref{eqn:ev_q1}--\eqref{eqn:ev_q2}.  
\end{itemize}
Using the above we are able to relate the geometry of the sets 
$$ \{(\bbb{1}(t,\du),\bbb{2}(t,\du))\that \du \in \LL{t_0}{t_1}\}\quad\text{where $t\in[t_0,t_1]$}$$ 
with the optimality of the sub-Riemannian trajectory $q(t)$.

\paragraph{The main result.}
 Leaving aside some technical complications (we need to divide the curve $q(t)$ into  pieces by excluding a finite number of points in the time interval $[t_0,t_1]$) it turns out that, if the curve $q(t)$ is an abnormal minimizing geodesics, then there exists a \emph{Pontryagin covector} $\wt\varphi(t)=(\varphi(t),0)\in \T_{(q(t),\Energy{t_0,t}[u])}^\ast (M\times\R)$ such that for each $t\in[t_0,t_1]$,  the set
$$S_t:=\{(\bbb{1}(t,\du),\<\bbb{2}(t,\du),\wt \varphi(t)>)\that \du\in \LL{t_0}{t_1}\}\subset \T_{(q(t),\Energy{t_0,t}[u])}(M\times\R)\times\R$$
is bounded by a graph of a quadratic map $\PP(t):\T_{(q(t),\Energy{t_0,t}[u])}(M\times\R)\times\T_{(q(t),\Energy{t_0,t}[u])}(M\times\R)\ra\R$ (see the figure below).

\begin{center}
\includegraphics[scale=0.7]{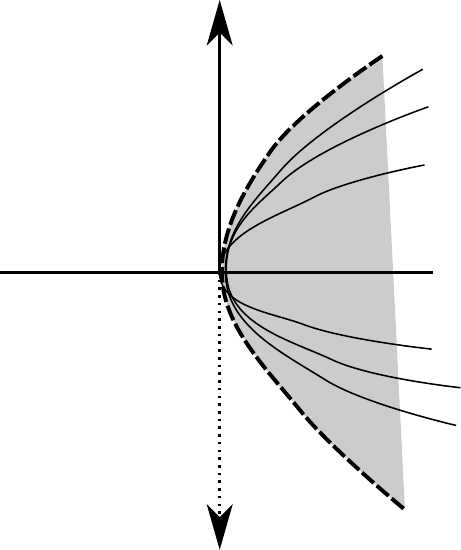}
\put(-270,-20){For a minimizing curve the second-order approximations of the image of the extended end-point}
\put(-270,-35){map (grey area)  must be  bounded by a graph of a quadratic function (dashed line). Solid lines re-} 
\put(-270,-50){present different quadratic approximations of the image of the extended end-point map $\END{t_0,t}{}$.}
\put(-140,100){$\END{t_0,t}{}[u]$}
\put(0,160){$\END{t_0,t}{}[u]+s\cdot \bbb{1}(t)+\frac{s^2}2 \cdot \bbb{2}(t)$}
\put(-75,170){$\R$-costs}
\put(-10,100){$M$-positions}
\end{center} 

Moreover, by taking this bounding map biggest possible, we may guarantee that $\PP(t)$ has properties of the value function in the sense of Bellman. In particular, it turns out that for every control $\du\in\LL{t_0}{t_1}$ the function
$$t\longmapsto \<\bbb{2}(t,\du),\wt \varphi(t)>-\PP(t)[\bbb{1}(t,\du),\bbb{1}(t,\du)]$$
is non-decreasing. The above properties are stated as Thm~\ref{thm:optimal_deg_2}, which is the main result of this paper.\smallskip

\paragraph{Discussion of the main result.}
By differentiation of the latter function with respect to $t$ (we were able to show that $\PP(t)$ is differentiable almost everywhere) we prove Theorem~\ref{thm:equations_of_motion}, obtaining a set of ODEs \eqref{eqn:xi_u_1}--\eqref{eqn:Phi_b_1} and an inequality \eqref{eqn:2_form_positive} that an abnormal minimizing SR geodesic should satisfy.

The derivation of these equations is, in our opinion, the most significant consequence of Thm~\ref{thm:optimal_deg_2}. (The well-known \emph{Goh condition} is a simple consequence of these equations -- see page~\pageref{par:Goh}.) The first of these equations is of the form
\begin{equation}
\label{eqn:main}
\<X_i\big|_{q(t)},\xi(t)>+2\ a_2(t)\cdot u_i(t)=0\ .
\end{equation}
It describes the controls $u_i(t)$ of the SR geodesic $q(t)$ in terms of vector fields $X_i$ spanning the distribution $\D$, and two functions $\xi(t)$, $a_2(t)$ related with the quadratic map $\PP(t)$. Note that if $a_2(t)\neq 0$ we may explicitly derive $u_i(t)$ as a function of other quantities. This observation is a basis of a distinction of abnormal extremals into sub-classes of \emph{2-normal} and \emph{2-abnormal extremals} proposed by us in Definition~\ref{def:2_normal_abnormal}. The former class is defined by the property that  $a_2(t)\neq 0$ for every $t\in[t_0,t_1]$, hence the controls $u_i$ are solutions of algebraic equations. Geometrically the 2-normal case corresponds to a situation when for every $t$ the graph of $\PP(t)$ strictly separates the cost direction from the set $S_t$. By contrast, in the 2-abnormal case the controls $u_i(t)$ are undetermined at some times $t\in[t_0,t_1]$ and the above separation is not strict. Schematically the difference is depicted on the figure below. 
\begin{center}
\includegraphics[scale=0.6]{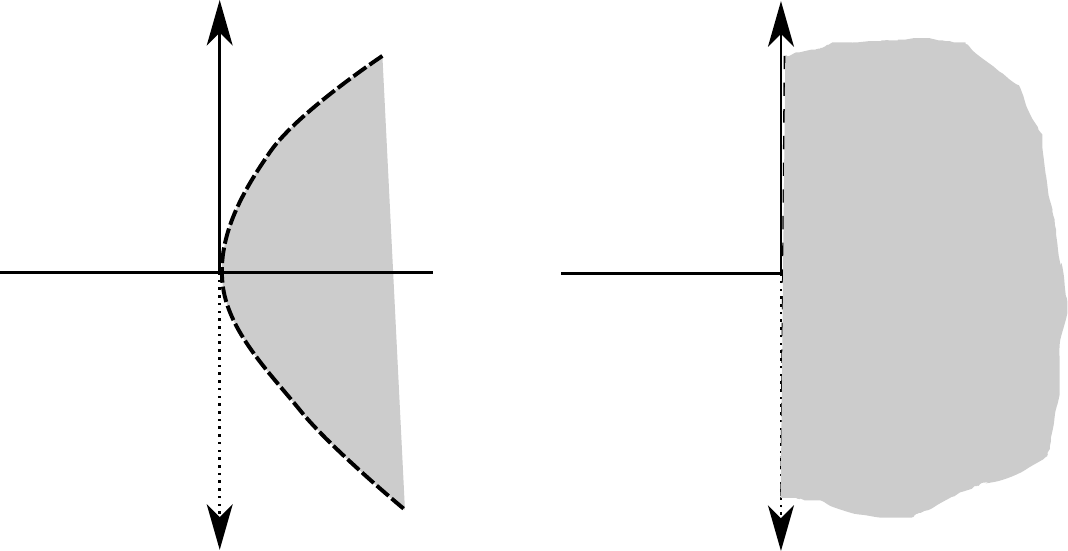}
\put(-75,150){$\R$}
\put(-235,150){$\R$}
\put(-195,130){$a_2(t)>0$}
\put(-55,130){$a_2(t)=0$}
\put(-300,-20){In the 2-normal case (left), the minimal 2-jet is separated from the cost direction $\R$,}
\put(-300,-35){while in the 2-abnormal case (right) there is no separation.}\label{fig:normal_abnormal}
\end{center}

The presence of equation \eqref{eqn:main} is very important from the point of view of the \emph{regularity problem of SR geodesics}, at least in the 2-normal case. Since, the controls $u_i(t)$, and the maps $\xi(t)$ and $a_2(t)$ are algebraically related, it is clear that the regularity problem is linked with the $t$-regularity of maps $\xi(t)$ and $a_2(t)$ or, more generally, the whole quadratic map $\PP(t)$. Unfortunately, in this paper we were only able to prove that these objects are of bounded variation, which is not enough for our purposes. However, I hope that other researchers, perhaps more experienced in optimal control theory, could help improving the regularity results. 

Let us note that the regularity of 2-abnormal extremals seems to be a much more difficult task requiring, in my opinion, study of third- or higher-order expansions of the extended end-point map. 

\paragraph{Supplementary results.}
The question whether the set $S_t$ is bounded is closely related with a special case of a linear-quadratic optimal control problem (we speak about a \emph{characteristic optimal control problem} of the SR trajectory $q(t)$). In our main result -- Thm~\ref{thm:optimal_deg_2} -- we used only the existence of a bound (and the fact that it is actually a graph of a quadratic map). An interesting problem on its own is the question whether the boundary points on $S_t$ are realizable by some controls from $\LL{t_0}{t}$, or to put it differently, does the characteristic optimal control problem has solutions. In general, this may not be the case as examples studied by us in Subsection~\ref{ssec:examples} show. However, when the solutions exists they can provide a lot of information about the SR trajectory $q(t)$, in particular prove its regularity. We discuss this topic in detail in Section~\ref{sec:suplementary}, providing both algebraic and differential (by means of the Pontryagin Maximum Principle) criteria for the existence of solutions of the characteristic  optimal control problem.  

In particular, we prove Lemma~\ref{lem:monti} in which we show that assumptions weaker then these of \cite[Thm 1.2]{Monti_third_order_2020} imply that if the abnormal minimizing SR geodesic $q(t)$ is a 2-normal extremal in the sense of Definition~\ref{def:2_normal_abnormal}, it must be piece-wise $C^2$ with a finite number of pieces.

\paragraph{Content of this paper.}
In the preliminary  Section~\ref{sec:prem} we state the sub-Riemannian geodesic problem and reformulate it in a control-theoretic setting. We also give a brief introduction to jets, define the extended end-point map, and recall first order optimality conditions (for normal and abnormal extremals).

In Section~\ref{sec:result} we state our main results -- Theorems~\ref{thm:optimal_deg_2} and \ref{thm:equations_of_motion} (linking them with the Agrachev-Sarychev theory). Later we discuss these results, in particular, addressing the regularity problem, Goh conditions, and introducing the notions of 2-normal and 2-abnormal curves. We also study two examples -- the Martinet system and a corner curve. 

Section~\ref{sec:proof} contains the proof of Theorem~\ref{thm:optimal_deg_2}. We begin it by some technical preparations in Subsection~\ref{ssec:char_ocp}, including construction of, so called, \emph{adapted coordinates}, and defining the \emph{characteristic control system} and the \emph{characteristic optimal control problem} related with a SR trajectory. The actual proof, contained in Subsection~\ref{ssec:prop_min_2_jet} is rather long, technical, and requires a few separate steps. 

In Section~\ref{sec:suplementary} we study additional information input on a SR geodesic if we can assure the existence of solutions of the characteristic optimal control problem. In particular we get some regularity results under this assumption.

We end by stating a few hypothesis in Section~\ref{sec:conclusion}.

\section{Preliminaries}\label{sec:prem}

\paragraph{Regularity of curves.} Throughout this work we will talk about curves on smooth manifolds of various classes of regularity. (Some emphasis will be put on improving the regularity of a curve by means of a bootstrap technique.) At the bottom level we have
$L^\infty$-curves, i.e.  \emph{measurable and essentially bounded}. A step higher are \emph{ACB} curves, i.e. \emph{absolutely continuous with $L^{\infty}$ derivatives}. The relation within these two is that if $\gamma(t)$ satisfies an ODE
$\dot \gamma(t)\overset{\text{a.e.}}= f(t)$, where $f(t)$ is of class $L^\infty$ then $\gamma(t)$ is ACB. This resembles the standard relation between $C^r$ and $C^{r+1}$ curves. Therefore it should not be very abusive to the standard mathematical jargon to denote the class $L^\infty$ as $C^{0_-}$, and the class ACB as $C^{1_-}$ (so that $C^{0_-+1}=C^{1_-}$ and also $C^{1_-}\subset C^0$). 

\subsection{The sub-Riemannian geodesic problem}\label{ssec:sr_geom}

\paragraph{The sub-Riemannian geodesic problem.} We shall be working within the following geometric setting. Let $M$ be a smooth $n$-dimensional manifold and let $\D\subset\T M$ be a smooth  distribution of rank $k$ on $M$. By $g(\cdot,\cdot):\D\times_M\D\ra\R$ we shall denote a positively-defined symmetric bilinear product on $\D$ (called a \emph{sub-Riemannian metric}). The triple $(M,\D,g)$ constitutes a \emph{sub-Riemannian} \emph{structure} (\emph{SR structure}, in short). Given a pair of points $q_0,q_1\in M$ and an interval $[t_0,t_1]\subset\R$ we consider the following \emph{SR geodesic problem}: 
\begin{equation*}\tag{P}
\label{eqn:sr_geod_prob}
\text{Find an absolutely continuous curve $q:[t_0,t_1]\ra M$ satisfying the following conditions:}
\end{equation*}
\begin{enumerate}[(P1)]
	\item\label{cond:1} $q$ is almost everywhere tangent to $\D$, i.e., $\dot q(t)\in \D_{q(t)}$ for a.e. $t\in[t_0,t_1]$.
	\item\label{cond:2} $q$ joins $q_0$ and $q_1$, that is, $q(t_0)=q_0$ and $q(t_1)=q_1$.
	\item in the set of all curves satisfying conditions (P\ref{cond:1}) and (P\ref{cond:2}), curve $q$ realizes the minimum of the \emph{energy functional} 
	$$q\longmapsto \Eng[q]:=\frac 12 \int_{t_0}^{t_1} g(\dot q(\tau),\dot q(\tau))\ \dd \tau\ .$$
\end{enumerate}
Ii is well-known \cite{Montgomery_2006} that under the assumption that $\D$ is bracket-generating, and that $M$ is complete, problem \eqref{eqn:sr_geod_prob} has a solution for every $q_0,q_1\in M$ and every $[t_0,t_1]\subset\R$. A solution of the problem \eqref{eqn:sr_geod_prob} is called a \emph{minimizing SR geodesic}. 

 It follows easily from the Cauchy-Schwarz inequality that the minimizing SR geodesic has to be normalized, i.e., $g(\dot q(t),\dot q(t))$ is constant on $[t_0,t_1]$. Therefore we may actually look for the solutions of \eqref{eqn:sr_geod_prob} in the class of ACB curves. In the forthcoming control-theoretic formulation of the SR geodesic problem, this simple observation will allow us to consider the smaller class of $L^\infty$-controls, instead of all $L^1$-controls. 
 


\paragraph{Control-theoretic formulation.} In a local version, problem \eqref{eqn:sr_geod_prob} can be formulated as an optimal control problem. Indeed, choose  $k$ smooth, linearly independent vector fields $X_1,\hdots, X_k$ spanning locally the distribution $\D$. Note that, by applying the Gram-Schmidt orthogonalization, we may assume that the fields $\{X_i\}_{i=1,\hdots,k}$ form an orthonormal basis of $\D$, i.e., $g(X_i,X_j)=\delta_{ij}$. Now every ACB curve $q(t)$ almost everywhere tangent to $\D$   may be regarded as a solution of
\begin{equation}
\label{eqn:trajectory_u}
\tag{$\Sigma$}
\dot q(t)=\sum_{i=1}^ku^i(t)\cdot X_i\Big|_{q(t)}\ ,
\end{equation}
where $t\mapsto u(t)=(u^1(t),\hdots, u^k(t))\in\R^k$ is an $L^\infty$-map, and equation \eqref{eqn:trajectory_u} is regarded in the sense of Caratheodory, i.e. the equality holds for almost every $t\in[t_0,t]$. We shall call such a $q$ a \emph{SR trajectory} corresponding to the \emph{control} $u$.

Note that under the assumption that the fields $X_i$ are normalized as $g(X_i,X_j)=\delta_{ij}$, the energy of such a trajectory is simply $\Energy{t_0,t_1}[u]$, where we define
\begin{equation}
\label{eqn:def_energy}
\tag{E}
    \Energy{t_0,t}[u]:=\frac 12\int_{t_0}^{t}\sum_{i=1}^k \left(u^i(\tau)\right)^2\dd \tau=\frac 12\|u\|_{\Ll{t_0}{t}}^2\ .
\end{equation}

\smallskip

From now on we shall \textbf{study the SR geodesic problem \eqref{eqn:sr_geod_prob} within such a local framework}, i.e. we take $M\overset{\text{loc.}}=\R^n$, and choose $\{X_i\}_{i=1,\hdots,k}$ a global $g$-orthonormal basis of smooth vector fields spanning $\D$. 

\begin{remark}[Notation in the extended configuration space]\label{rem:notation}
It is convenient to think of both, the trajectory $q(t)$, and the energy $\Energy{t_0,t}$ at the same time, by introducing \emph{the extended configuration space} $\wt M:=M\times\R$. For a control $u\in \LL{t_0}{t_1}$ the \emph{extended trajectory} is defined as $\bm q(t):=(q(t),\Energy{t_0,t}[u])$, where $q(t)$ satisfies \eqref{eqn:trajectory_u}. Similar objects, consequently denoted by bold letters, will be considered thorough the paper. 
\end{remark}

\subsection{On jets}\label{ssec:jets}

In the paper we speak about second-order optimality conditions in SR geometry. The term \emph{second-order} means that we  will consider the second-order Taylor expansions of various objects. Geometrically, we rather speak about \emph{2-jets}. Below we recall some basic definitions and constructions from jet theory \cite{Saunders_1989, KMS_1993}. 

\paragraph{Jets of curves.} We say that two $C^2$-curves $\gamma,\widetilde\gamma:(-\eps,\eps)\lra M$ are \emph{tangent of order two} at $p=\gamma(0)=\widetilde{\gamma}(0)$ if in some (and thus any) local coordinate system around $p$ on $M$ we have 
$$\gamma'(0)=\widetilde\gamma'(0)\quad\text{and}\quad \gamma''(0)=\widetilde\gamma''(0)\ .$$
In other words, maps $s\mapsto\gamma(s)$ and $s\mapsto\widetilde\gamma(s)$ have the same order-two Taylor expansions at $s=0$. The relation $\sim_2$ of being tangent of order two  is an equivalence relation and its equivalence classes $[\gamma]_{\sim_2}=[\widetilde\gamma]_{\sim_2}$ are called \emph{2-jets} at $p$. We denote the set of all 2-jets at $p$ by $\T^2_p M$, and call it the \emph{second tangent space of $M$ at $p$}. In fact the collection $\T^2M:=\bigcup_{p\in M}\T^2_pM$ is a locally trivial bundle over $M$. We call it the \emph{second tangent bundle of $M$}. It is an example of a \emph{graded bundle} in the sense of \cite{Grabowski_Rotkiewicz_2012}, with the canonical multiplicative $\R$-action defined on representatives by $\lambda \mapsto \left(s\mapsto \gamma(\lambda\cdot s)\right)$. 
\smallskip

Given a curve $\gamma:(-\eps,\eps)\ra M$ and its (local) second-order Taylor expansion  
$$\gamma(s)\overset{loc}=p+s\cdot \beta^{(1)}+{s^2}\cdot \beta^{(2)}+o(s^2)\ ,$$
we may thus identify its 2-jet with a pair $(\beta^{(1)},\beta^{(2)})\simeq[\gamma]_{\sim_2}$. Although one should remember, that only the first term $\beta^{(1)}$ is a geometric object, representing a tangent vector at $p$, while the term $\beta^{(2)}$ on its own (without the presence of a particular $\beta^{(1)}$) has no intrinsic geometric meaning. In a particular situation when $M=V$ is a vector space we may, however, canonically identify $\T^2_p V\simeq V\times V$, by considering the above Taylor expansion in any linear coordinates on $V$. Thus for a vector space, the bare term $\beta^{(2)}$ has a geometric interpretation of an element of $V$. 

\paragraph{Jets of maps.}

Let $F:M\ra N$ be a smooth map between manifolds. For a  $C^2$-curve $\gamma:(-\eps,\eps)\ra M$ passing through $\gamma(0)=p\in M$ consider a (local) second-order  Taylor expansion 
$$F(\gamma(s))\overset{loc}=F(p)+s\cdot \bb{1}+s^2\cdot \bb{2}+o(s^2)\ .$$
 Then the pair $(\bb{1},\bb{2})$ depends only on the 2-jet $[\gamma]_{\sim_2}$. Thus we may define the 2-jet of map $F$ $$ \T^2F:\T^2 M\lra \T^2N\qquad\text{via}\qquad \T^2_pM\ni[\gamma]_{\sim_2}\longmapsto [F(\gamma)]_{\sim_2}\in\T^2_{F(p)}N\ .$$ 
$\T^2 F$ is a bundle map over $F:M\lra N$.

\paragraph{Pairings on jets.}
As a particular example of the above consider a smooth function $f:M\lra \R$ and a $C^2$-curve $\gamma:(-\eps,\eps)\lra M$ passing through $\gamma(0)=p\in M$. The order-two Taylor expansion of the composition $f(\gamma):(-\eps,\eps)\lra\R$ reads as
$$f(\gamma(s))=f(p)+s\cdot \Der_pf[\gamma'(0)]+s^2\cdot\left\{ \Der_pf[\gamma''(0)]+\Der^2_pf[\gamma'(0),\gamma'(0)]\right\}+o(s^2)\ ,$$
and it depends only on the 2-jet $[\gamma]_{\sim_2}$. Further, the resulting 2-jet $[f(\gamma)]_{\sim_2}$ is valued in the second tangent space of a vector space $\T^2_{f(p)}\R\simeq \R\times\R$, and so the term of degree two 
\begin{equation}\label{eqn:pairing_jet}
\Der_pf[\gamma''(0)]+\Der^2_pf[\gamma'(0),\gamma'(0)]
\end{equation}
has an intrinsic geometric meaning. 
\smallskip

Note that for a fixed 2-jet $[f(\gamma)]_{\sim_2}$, the value of \eqref{eqn:pairing_jet} depends only on the first and second derivatives $\Der_pf$ and $\Der^2_p f$, i.e. the second-order Taylor expansion of $f$ at $p$. Since we may construct smooth functions with arbitrary derivatives at a point, we may define via
\begin{equation}\label{eqn:pairing_jet_new}
    \<(\phi,\Phi),(\bb{1},\bb{2})>:=\<\phi,\bb{2}>+\Phi[\bb{1},\bb{1}]
\end{equation}
a natural (and geometrically meaningful) pairing between a 2-jet of a curve $(\bb{1},\bb{2})\in \T^2_pM$ and a pair $(\phi,\Phi)$ consisting of a linear map $\phi:\T_p M\lra\R$ and  a symmetric bi-linear map $\Phi:\T_pM\times\T_pM\lra\R$.

\begin{remark}\label{rem:2_covector_pairing}
It is worth to mention that the pairs $(\phi,\Phi)$ as above can be identified with the elements of $\T^{2\ast}_pM$, i.e. 2-covectors at $p\in M$ in the sense of Tulczyjew \cite{Tulczyjew}. We define the latter as equivalence classes of smooth functions on $M$, saying that $f,\widetilde f:M\ra\R$ are equivalent if and only if $f(\gamma(s))$ and $\widetilde f(\gamma(s))$ have the same second-order Taylor expansions at $s=0$ for any smooth curve $\gamma:(-\eps,\eps)\ra M$ passing through $p=\gamma(0)$. The pairing \eqref{eqn:pairing_jet_new} can be thus interpreted as a map $\T^{2\ast}_{p}M\times \T^{2}_{p} M\lra \R$. We remark that it has an elegant description in terms of the Weyl algebra of dual numbers \cite{KMS_1993}.
\end{remark}

\subsection{The end-point map and its second-order expansion}\label{ssec:end-point_map}

\paragraph{The end-point  map and its second expansion.} Denote by 
$$\End{t_0,t}{q_0}:\LL{t_0}{t}\lra M$$ 
the \emph{end-point map} of the control system \eqref{eqn:trajectory_u} at time $t$  and with the initial condition $q_0$ at $t=t_0$. That is, for a control $u\in\LL{t_0}{t}$, $\End{t_0,t}{q_0}[u]$ is a solution $q(t)$ of \eqref{eqn:trajectory_u} satisfying $q(t_0)=q_0$.
\smallskip

We would like to study the following second-order expansion of this map (written in some local coordinate system)
\begin{equation}
\label{eqn:expansion_end}
    \End{t_0,t}{q_0}[u+s\cdot\du]\overset{loc}=\End{t_0,t}{q_0}[u]+s\cdot \bb{1}(t,\du)+s^2\cdot \bb{2}(t,\du)+o(s^2)\ ,
\end{equation}
where $\du\in \T_u \LL{t_0}{t}\simeq \LL{t_0}{t}$ is a control, and $s$ is a real parameter. Of course $\bb{1}(t,\du)=\Der_u\End{t_0,t}{q_0}[\du]$, and one is tempted to write $\bb{2}(t,\du)=\frac 12\ \Der^2_u\End{t_0,t}{q_0}[\du,\du]$. The last equality makes, however, no sense as in general the second derivative of a manifold-valued map is not a well-defined geometric object -- cf. \cite[Chap. 20]{Agrachev_Sachkov_2004}. Expansion \eqref{eqn:expansion_end} has a geometric meaning though -- by our considerations in Ssec.~\ref{ssec:jets}  it describes the second tangent map $\T^2\End{t_0,t}{q_0}$ evaluated on a 2-jet $[s\mapsto u+s\cdot\du+s^2\cdot 0]_{\sim 2}\in \T^2_u\LL{t_0}{t}\simeq \LL{t_0}{t}\times\LL{t_0}{t}$. Thus a pair $(\bb{1}(t,\du),\bb{2}(t,\du))$ is a local representative of some element in $\T^2_{q(t)}M$. 
\smallskip

To calculate $\bb{1}(t,\du)$ and $\bb{2}(t,\du)$, denote by $q_s(t)$ a solution of the following ODE
\begin{equation}
\label{eqn:end_u_plus_s_du}
    \dot {q_s}(t)=\sum_{i=1}^k\left(u^i(t)+s\cdot\du_i(t)\right)\cdot X_i\Big|_{q_s(t)}\qquad\text{with}\qquad q_s(t_0)=q_0\ .
\end{equation}
Obviously $q_s(t)=\End{t_0,t}{q_0}[u+s\cdot\du]$. What is more, \eqref{eqn:end_u_plus_s_du} is an ODE in the sense of Caratheodory, depending smoothly on the parameter $s$. As such it has a solution which also smoothly depends on $s$ -- see \cite{Bressan_Piccoli_2004} -- and, moreover, equations for $\bb{1}(t,\du)=\pa_s\big|_{s=0}q_s(t)$ and $\bb{2}(t,\du)=\frac 12 \pa_s^2\big|_{s=0}q_s(t)$ are obtained by differentiating \eqref{eqn:end_u_plus_s_du} with respect to $s$, leading to (for notation simplicity below we omit the dependence of $\bb{1}(t,\du)$ and $\bb{2}(t,\du)$ on $\du$) 
\begin{equation}\label{eqn:b1}
\dot{\bb{1}}(t)=\sum_{i=1}^k u_i(t)\cdot \Der_{q(t)} X_i\left[\bb{1}(t)\right]+\sum_{i=1}^k \du_i(t)\cdot X_i\Big|_{q(t)}
\end{equation}
with the initial condition $\bb{1}(t_0)=0$; and
\begin{equation}\label{eqn:b2}
\dot{\bb{2}}(t)=\sum_{i=1}^k u_i(t)\left\{\Der_{q(t)}X_i\left[\bb{2}(t)\right]+ \frac 12\Der^2_{q(t)}X_i \left[\bb{1}(t),\bb{1}(t)\right]\right\}+2\sum_{i=1}^k \du_i(t)\cdot\Der_{q(t)} X_i\left[\bb{1}(t)\right]
\end{equation}
where $\bb{1}(t)$ is as above, and the initial condition is $\bb{2}(t_0)=0$. 

\paragraph{The extended end-point map.} Now we would like to add the energy \eqref{eqn:def_energy} to the picture. We define the
\emph{extended end-point map} 
$$  \END{t_0,t}{q_0}: \LL{t_0}{t}\lra M\times\R\ , $$ simply as $\END{t_0,t}{q_0}[u]:=(\End{t_0,t}{q_0}[u],\Energy{t_0,t}[u])$. For any $\du\in\LL{t_0}{t}$ the second expansion of the energy term is just 
\begin{equation}
    \label{eqn:expansion_energy}
    \Energy{t_0,t}[u+s\cdot\du]=\Energy{t_0,t}(u)+s\cdot \cc{1}(t,\du)+s^2 \cdot \cc{2}(t,\du)\ ,
\end{equation}
where $\cc{1}(t,\du)=\int_{t_0}^t \sum_i u_i(\tau)\cdot \du_i(\tau)\ 
\dd \tau$ and $\cc{2}(t,\du)=\int_{t_0}^t \sum_i \du_i(\tau)^2\ \dd \tau$. Or, in differential terms:
\begin{equation}
    \label{eqn:c1_c2}
    \begin{split}
\dot {\cc{1}}(t)&=\sum_{i=1}^k u_i(t)\cdot \du_i(t)\\
\dot {\cc{2}}(t)&=\sum_{i=1}^k  \du_i(t)^2
\end{split}
\end{equation}
with the initial conditions $\cc{1}(t_0)=0$ and $\cc{2}(t_0)=0$.

\paragraph{Normal and abnormal extremals}
Note that if for some $t$ the vector space 
$$\image\Der_u \END{t_0,t}{}=\{(\bb{1}(t,\du),\cc{1}(t,\du))\that \du \in \LL{t_0}{t_1}\}$$
equals  $\T_{(q(t),\Energy{t_0,t}[u])} (M\times\R)$ then, by the open-mapping theorem, the map $\END{t_0,t}{}$ is open at $u$, and hence the related trajectory cannot be a minimizing sub-Riemannian geodesic. In this way we obtained a necessary condition that a SR geodesics should satisfy. Curves that satisfy this conditions are known as \emph{extremals}. We can further classify these according to the geometric properties of $\image\Der_u \END{t_0,t}{}$.  

\begin{definition}
\label{def:extremals}
Consider a  sub-Riemannian trajectory  $q(t)$ with $t\in[t_0,t_1]$ corresponding to the control $u\in \LL{t_0}{t_1}$. Curve $q(t)$ is called an \emph{extremal} if $\image\Der_u \END{t_0,t}{}$ is a proper subspace of $\T_{(q(t_1),\Energy{t_0,t_1}[u])} (M\times\R)$. An extremal is called \emph{normal} if the space $\image\Der_u \END{t_0,t}{}$ does not contain the subspace $\{0\}\times\R$. If the image $\image\Der_u \END{t_0,t}{}$ is contained in a space $V\times\R$, where $V\subset \T_{q(t_1)}M$ is some proper subspace, $q(t)$ is called \emph{abnormal}. We will say an abnormal extremal has \emph{co-rank $r$} if $\image\Der_u \END{t_0,t}{}$ is a co-rank $r$-subspace of $\T_{(q(t_1),\Energy{t_0,t_1}[u])} (M\times\R)$.\smallskip

\emph{Strictly abnormal} extremals are abnormal extremals which are not normal, that is $\image\Der_u \END{t_0,t}{}=V\times\R$ for some proper subspace $V\subset \T_{q(t_1)}M$.
\end{definition}

Normal extremals are known to be $C^\infty$-smooth and energy-minimizing on small distances \cite{Liu_Sussmann_1995,Montgomery_2006}.  

\paragraph{First-order abnormal optimality conditions}

Let us recall the following standard fact \cite{Montgomery_2006}.

\begin{lemma} Let a SR trajectory $q(t)$ corresponding to the control $u\in\LL{t_0}{t_1}$ be an abnormal extremal. Then there exists a nowhere-vanishing covector curve $\varphi(t)\in \T^\ast_{q(t)} M$ satisfying the following conditions:
\begin{equation}\label{eqn:pontryagin_covector}
    \begin{split}
    &\<X_i\big|_{q(t)},\varphi(t)>=0\qquad \text{for every $i=1,2,\hdots,k$}\\ 
    &\<b,\dot\varphi(t)>+\<\sum_iu_i(t)\cdot\Der_{q(t)}X_i[b],\varphi(t)>=0\quad\text{for every $b\in\T_{q(t)}M$. }
      \end{split}
\end{equation}
  We shall call such a $\varphi(t)$ a \emph{Pontryagin covector}. 
\end{lemma}

It is instructive to see that conditions \eqref{eqn:pontryagin_covector} can be easily derived from the following one 
$$\image\Der_u\END{t_0,t}{}\subset \ker (\varphi(t),0)\quad\text{(equivalently}\quad \<\bb{1}(t,\du),\varphi(t)>=0 \text{ for every $t$ and $\du$)}\ ,  $$
which express the fact that $\image\Der_u\END{t_0,t}{}$ is contained in a proper subspace $V\times\R\subset \T_{q(t)}M\times\T_{\Energy{t_0}{t}[u]}\R$ for every $t$.

 To see this it is enough to differentiate the above condition with respect to $t$ to get:
\begin{align*}
0=&\<\dot{\bb{1}}(t,\du),\varphi(t,\du)>+\<\bb{1}(t,\du),\dot{\varphi}(t)>\overset{\eqref{eqn:b1}}=\\
&\<\sum_i u_i(t)\Der_{q(t)} X_i[\bb{1}(t,\du)]>+\sum_i \du_i(t)\<X_i\big|_{q(t)},\varphi(t)>+ \<\bb{1}(t,\du),\dot{\varphi}(t)>\ .
\end{align*}
This should hold for every possible choice of the control $\du$. Thus the part linear in $\du$ should vanish leading to $\<X_i\big|_{q(t)},\varphi(t)>=0$ for every $i=1,2,\hdots,k$. Consequently we are left with
\begin{align*}
0=\<\sum_i u_i(t)\Der_{q(t)} X_i[\bb{1}(t,\du)]>+ \<\bb{1}(t,\du),\dot{\varphi}(t)>\ .
\end{align*}\smallskip

It is instructive to see that the notions from Definition~\ref{def:extremals} give the well-known first order optimality condition. For example if $q(t)$ is an abnormal extremal then there exists a non-zero covector $\varphi_0\in(\R^n)^\ast$ such that $\varphi_0+0\cdot\dd r\in (\R^{n+1})^\ast$ annihilates the reachable set $\RR(t_0,t_1)$. In more concrete terms 
$\<\q{1}(t_1,\du),\varphi_0>=0$ for every $\du\in\LL{t_0}{t_1}$. By \eqref{eqn:ev_q1} this gives us
\begin{align*}
    0=\<\int_{t_0}^{t_1}\sum_i \du_i(t)\cdot \Y{i}{1}(t)\ \dd t,\varphi_0 >=\int_{t_0}^{t_1}\sum_i \du_i(t)\cdot \<\Y{i}{1}(t),\varphi_0>\ \dd t
\end{align*}
and so for every $t\in [t_0,t_1]$ and every $i=1,2\hdots,k$
\begin{align*}
0=\<\Y{i}{1}(t),\varphi_0>\overset{\text{Lem.~\ref{lem:adapted}}}=\< \lin{1}(t)\left[X_i(q(t))\right],\varphi_0>=\<X_i(q(t)),\lin{1}(t)^\ast\varphi_0>\ .
\end{align*}
As we see $\varphi(t):=\lin{1}(t)^\ast\varphi_0\in \T^\ast_{q(t)}M$ annihilates the whole distribution $\D$ along the trajectory $q(t)$. Note that thanks to the evolution equations of $\lin{1}(t)$,  $\varphi(t)$ defined in such a way is a \emph{Pontryagin covector}, i.e. it is subject to an ODE\label{pontryagin_covector}
$$\<b,\dot\varphi(t)>+\<\sum_iu_i(t)\cdot\Der_{q(t)}X_i[b],\varphi(t)>=0\quad\text{for every $b\in\T_{q(t)}M$. }$$
At the end of the day we arrived at the standard conditions for an abnormal extremal \cite{Montgomery_2006}.

\section{The main result and its discussion}\label{sec:result}

In this section we will formulate our main result -- Theorem~\ref{thm:optimal_deg_2} -- and discuss its consequences. We would like to present our result as an extension of the classical Agrachev-Sarychev theory. However, for the sake of generalization,  we prefer to use the language of 2-jets rather then the one of first and second derivatives originally used in \cite{Agrachev_Sarychev_1996}. For this reason our notation will be somehow non-standard. 

\subsection{The Agrachev-Sarychev Index Lemma}

\paragraph{Statement of the Index Lemma}  Recall that by a \emph{negative index} of a quadratic form we understand the maximal dimension of a subspace on which this form is negatively-defined. We have the following classical result. 

        \begin{lemma}[Index Lemma \cite{Agrachev_Sarychev_1996}]\label{lem:AS_index} Let $q(t)$, with $t\in[t_0,t_1]$, be an abnormal minimizing SR geodesic of co-rank $r$, corresponding to the control $u\in\LL{t_0}{t_1}$.  Then there exists a non-vanishing Pontryagin covector $\varphi(t)\in\T^\ast_{q(t)} M$ such that for each $t\in[t_0,t_1]$ we have 
        \begin{enumerate}[(i)]
        \item $\operatorname{Im}{\Der_u\END{t_0,t}{q_0}}\subset\ker \varphi(t)$, and \item the negative index of the quadratic form (the \emph{Agrachev-Sarychev index}, \emph{AS index} in short)
        $$\< \Der^2_u\End{t_0,t}{q_0}[\cdot,\cdot],\varphi(t)>:\ker \Der_u\END{t_0,t}{q_0}\times \ker \Der_u\END{t_0,t}{q_0}\lra\R$$
        is smaller than $r$. 
        \end{enumerate}
        
        Above we identify $\varphi(t)\in \T^\ast M$ with $\varphi(t)+0\cdot \dd r\in \T^\ast M\times\T^\ast \R$, so that the pairing between $\varphi(t)$ and the element $\Der_u\END{t_0,t}{q_0}[\du]$ makes sense. 
        \end{lemma}
For a concise proof we refer to the recent monograph \cite[Appendix B.]{Rifford_2014}. Let us remark, that although the second derivative $\Der^2_u\End{t_0,t}{}$ has only a local sense, the above map $\< \Der^2_u\End{t_0,t}{q_0}[\cdot,\cdot],\varphi(t)>$ is well-defined, since our attention is restricted to the kernel $\ker \Der_u\End{t_0,t}{}$ -- cf. our discussion of pairing of 2-jets preceding formula \eqref{eqn:pairing_jet_new}. 

\paragraph{A simple corollary}
We are most interested in the situation when the Agrachev-Sarychev index is actually zero, i.e. the respective quadratic map is non-negatively defined. It turns out that this is not so difficult to achieve by the following argument. 

As we shall prove later (see Lemma~\ref{lem:AS_index_zero_1}), function $t\longmapsto \operatorname{ind}_-\< \Der^2_u\End{t_0,t}{q_0}[\cdot,\cdot],\varphi(t)>$, assigning to each $t\in[t_0,t_1]$ the respective AS index, is non-decreasing. On the other hand, by the result of Lemma~\ref{lem:AS_index} this function is bounded. It follows that there are at most $r-1$ points of growth on the interval $[t_0,t_1]$. Inbetween these points the index must be zero. Therefore for the price of excluding a finite number of points from the trajectory (or, to see it differently, of dividing the trajectory into a finite number of pieces), we may set the AS index to zero. Now we have
    \begin{lemma}\label{lem:AS_index_zero} Under the assumptions of Lemma~\ref{lem:AS_index}, there exists a non-vanishing Pontryagin covector $\varphi(t)\in\T^\ast_{q(t)} M$ such that $$\<\bb{1}(t,\du),\varphi(t)>=0$$ 
  for every $t\in[t_0,t_1]$ and every $\du\in\LL{t_0}{t_1}$. 
  \smallskip
  
  Further, there exist at most $r+1$ points $\tau_0=t_0\leq \tau_1\leq\tau_2\leq\hdots\leq \tau_{r'}=t_1$ such that for each subinterval $[t_0',t_1']\subset(\tau_j,\tau_{j+1}]$ the following inequality
        \begin{equation}\label{eqn:inequality_as_lemma}
        \<\bb{2}(t,\du),\varphi(t)>\geq 0
        \end{equation}
        holds for every $\du\in\ker \Der_u\END{t_0',t_1'}{}\subset\LL{t_0'}{t_1'}$ and every $t\in[t_0',t_1']$. 
        \end{lemma}

\subsection{Optimality conditions of degree 2} 
\paragraph{The main result}
Our main result -- Theorem~\ref{thm:optimal_deg_2} -- may be seen as a twofold generalization of Lemma~\ref{lem:AS_index_zero}. First of all, we generalize inequality \eqref{eqn:inequality_as_lemma} by adding a quadratic part to it. Thanks to that controls $\du$ can be arbitrary, not just form the kernel of $\Der_u\END{t_0,t}{}$ as was previously the case. Since we are no longer restricted in the choice of controls $\du$, we will be able to add the monotonicity condition \eqref{eqn:2_jet_monotone} related with the discussed  inequality. This is the part that was completely absent in the original Agrachev-Sarychev theory. It has, however, important consequences, allowing us to derive new equations for abnormal geodesic in 
Theorem~\ref{thm:equations_of_motion}.

        \begin{theorem}\label{thm:optimal_deg_2} Let $q(t)$, with $t\in[t_0,t_1]$, be a strictly abnormal minimizing SR geodesic of co-rank $r$, corresponding to the control $u\in\LL{t_0}{t_1}$. For $\du\in \LL{t_0}{t_1}$ consider the second-order expansion of the extended end-point map in the form
        \begin{align*}\END{t_0,t}{}&[u+s\cdot \du]=\left(\End{t_0,t}{}[u+s\cdot\du],\Energy{t_0,t}[u+s\cdot\du]\right)\overset{\text{loc.}}=\\
        &\left(q(t),\Energy{t_0,t}[u]\right)+s\left(\bb{1}(t,\du),\cc{1}(t,\du)\right)+s^2\left(\bb{2}(t,\du),\cc{2}(t,\du)\right)+o(s^2)\ .
        \end{align*}
        
        Then there exists a non-vanishing Pontryagin covector $\varphi(t)\in\T^\ast_{q(t)} M$ such that 
        \begin{equation}\label{eqn:equality_b1}
            \<\bb{1}(t,\du),\varphi(t)>=0
        \end{equation}
        for every $t\in[t_0,t_1]$ and every $\du\in\LL{t_0}{t_1}$. 
        \smallskip

    Further there exists a division of the interval $[t_0,t_1]$ by at most $r(n-r-k+1)+1$ points $\tau_0=t_0 < \tau_1<\hdots<\tau_N=t_1$ such that on each sub-interval $[t_0',t_1']\subset(\tau_i,\tau_{i+1})$ there exists a family of symmetric  bi-linear maps $\PP(t):\T_{\END{t_0',t}{}[u]}(M\times\R)\times\T_{\END{t_0',t}{}[u]}(M\times\R)\lra \R$ parameterized by $t\in[t_0',t_1']$ satisfying the following conditions:
    \begin{itemize}
        \item $-\PP(t)[(\bb{1},\cc{1}),(\bb{1},\cc{1})]$ equals to the infimum
        $$\inf\{\<\bb{2}(t,\du),\varphi(t)>\that \du\in \LL{t_0'}{t_1'}, \bb{1}(t,\du)=\bb{1}, \cc{1}(t,\du)=\cc{1} \}\ .$$
            We conclude that for every $\du\in\LL{t_0'}{t_1'}$ and $t\in[t_0',t_1']$ the following inequality holds
        \begin{equation}\label{eqn:inequality_full}
            \<\bb{2}(t,\du),\varphi(t)>+\PP(t)[(\bb{1}(t,\du),\cc{1}(t,\du)),(\bb{1}(t,\du),\cc{1}(t,\du))]\geq 0;
        \end{equation}
    
        \item the assignment $t\mapsto \PP(t)$ is a map of bounded variation, hence differentiable almost everywhere,
        \item moreover,  for every $\du\in\LL{t_0'}{t_1'}$ the assignment
        \begin{equation}\label{eqn:2_jet_monotone}
            t\longmapsto \<\bb{2}(t,\du),\varphi(t)>+\PP(t)[(\bb{1}(t,\du),\cc{1}(t,\du)),(\bb{1}(t,\du),\cc{1}(t,\du))]
        \end{equation}
        is non-decreasing.
    \end{itemize}
    \end{theorem}
    
The proof is rather long and technical, and will be given in Section~\ref{sec:proof}.

\paragraph{Explanation of the main result}
Let us briefly discuss  the above result. At the fundamental level it says that there is a relation between optimality of an abnormal curve and the geometry of the second-order expansion of the extended end-point map $\END{t_0,t}{}$ for each time $t$. This relation is expressed by \eqref{eqn:equality_b1} -- an equality involving the first derivatives  of this map (the standard first-order optimality condition) -- and \eqref{eqn:inequality_full} -- an inequality  involving 2-jets of $s\mapsto \END{t_0,t}{}[u+s\cdot\du]$. The latter inequality is stated by means of a pairing of a 2-jet and 2-covector $((\phi(t),0),\PP(t))$ on $M\times\R$ -- cf. Remark~\ref{rem:2_covector_pairing}. To understand the meaning of these conditions let $f:M\times\R\ra\R$ be any function realizing this 2-covector, i.e such that $\Der f[\cdot]=\<\cdot,(\varphi(t),0)>$ and $\Der^2f[\cdot,\cdot]=\PP(t)[\cdot,\cdot]$. Then conditions \eqref{eqn:equality_b1}--\eqref{eqn:inequality_full} imply that, up to the terms of order higher than 2, the curves $s\mapsto \End{t_0,t}{}[u+s\cdot \du]$ lie on a one side of a level set of $f$ -- cf. our considerations on page~\pageref{eqn:pairing_jet}. To put it differently: there is a quadratic hypersurface which bounds the 2-jets of curves $s\mapsto \End{t_0,t}{}[u+s\cdot \du]$ for every possible choice of control $\du$. Since $\Der f=(\varphi(t),0)$ is trivial on $\R$-component, this hypersurface is 1-tangent to the energy direction $\R$ in $M\times\R$. The question whether we have also 2-tangency of these objects will be a basis of a distinction between \emph{2-normal} and \emph{2-abnormal} curves, which we shall introduce later in this section. 
\smallskip

 At this point it is worth to explain a presumable contradiction, pointed out by one of the reviewers of this paper. Namely, in the Agrachew-Sarychew Lemma \ref{lem:AS_index}, since our attention is restricted to controls $\du$ in the kernel of $\Der_u \END{t_0,t_1}{q_0}$, we can state conditions on the second derivative of the end-point map $\Der^2_u\END{t_0,t_1}{q_0}$ -- which now becomes the dominant term of the expansion of $\END{t_0,t_1}{q_0}[u+s\cdot\du]$. On the contrary, assumptions of Theorem~\ref{thm:optimal_deg_2} do not require $\Der_u\END{t_0,t_1}{q_0}[\du]$ to vanish, hence the access to the second derivative $\Der^2_u\END{t_0,t_1}{q_0}[\du,\du]$ does not seem so obvious. Note, however, that we do not study the end-point map  $\END{t_0,t_1}{q_0}[u+s\cdot\du]$ directly, but rather through the composition $f\circ\END{t_0,t_1}{q_0}[u+s\cdot\du]$ . The latter expands as (cf. the last paragraph of Subsection~\ref{ssec:jets})
\begin{align*}
    f\circ&\END{t_0,t_1}{q_0}[u+s\cdot\du]= f\circ\END{t_0,t_1}{q_0}[u]+s\cdot \Der f\left[\Der_u\END{t_0,t_1}{q_0}[\du]\right]+\\
    &s^2\cdot\left(\Der f\left[\Der^2_u\END{t_0,t_1}{q_0}[\du,\du]\right]+\Der^2f\left[\Der_u\END{t_0,t_1}{q_0}[\du],\Der_u\END{t_0,t_1}{q_0}[\du]\right]\right)+o(s^2)
\end{align*}
Now $\Der f$ is the Pontryagin covector which annihilates $\Der_u\END{t_0,t_1}{q_0}[\du]$, making the quadratic part the dominant term in the expansion. Theorem~\ref{thm:optimal_deg_2} gives us precisely information about this quadratic term, for $\Der f=(\varphi(t),0)$ and $\Der^2 f[\cdot,\cdot]=\PP(t)[\cdot,\cdot]$.\smallskip

The important part of the assertion is that we have a direct description of what this quadratic map $\PP(t)$ is. Namely, $-\PP(t)[(\bb{1},\cc{1}),(\bb{1},\cc{1})]$ is defined as a solution of the (optimal control) problem of minimizing the value $\<\bb{2}(t,\du),\varphi(t)>$ in the class of controls $\du$ that give  fixed values  $\bb{1}(t,\du)=\bb{1}$ and $\cc{1}(t,\du)=\cc{1}$. In other words, we cannot further improve inequality \eqref{eqn:inequality_full} -- it describes the envelope of the space of all 2-jets of $s\mapsto \END{t_0,t}{}[u+s\cdot\du]$. This description allows to calculate $\PP(t)$ in concrete examples. Let us  stress, that neither the finiteness of the solution of the above optimal control problem, nor the property it is actually quadratic in $(\bb{1},\cc{1})$ is a trivial fact. 
\smallskip

There is a caveat to what we said above. Namely the described properties of the expansion of the extended end-point map may not be true on the whole geodesic, but only on its sub-intervals obtained after a removal of a finite number of points. These points play essentially the same role as the points at which the AS index grows (cf. Lemma~\ref{lem:AS_index})
-- their removal guarantees the positive-definiteness of the quadratic maps considered in the discussed results. In principle, it should be possible to reformulate the conditions \eqref{eqn:inequality_full} of Theorem~\ref{thm:optimal_deg_2} as a finite-index condition for some global map quadratic in $\du$ (as it is done in the original Agrachev-Sarychev Lemma~\ref{lem:AS_index}). However, in such a formulation one would not be able to state condition \eqref{eqn:2_jet_monotone}. As we shall see shortly, the latter condition is responsible for the derivation of evolution equations, hence possibly also proving regularity of the trajectory. Therefore, we see no satisfactory way of getting rid of the point-removal procedure. On the other hand, isolated point singularities may be treated by the methods of \cite{no_corners_2016}. 
    
\paragraph{Equations of an abnormal geodesic} 
Conditions \eqref{eqn:equality_b1}--\eqref{eqn:inequality_full} are formulated for every $t\in[t_0,t_1]$. On the other hand, the elements $\bb{i}(t,\du)$ and $\cc{i}(t,\du)$, for $i=1,2$ encoding the 2-jet of $s\mapsto\END{t_0,t}{}[u+s\cdot\du]$ satisfy \eqref{eqn:b1}-\eqref{eqn:b2} and \eqref{eqn:c1_c2}, and thus can be constructed by an integration along $[t_0,t]$. Therefore, we should expect that optimality conditions at time $t$ are somehow \emph{compatible} with the optimality conditions at all preceding times. \smallskip

For condition \eqref{eqn:equality_b1} this ``compatibility'' is expressed by the fact that $\varphi(t)$ is a Pontryagin covector -- we have seen that at the very end of Section~\ref{sec:prem}. 
By the same token, condition \eqref{eqn:2_jet_monotone} is responsible for ``compatiblity'' of conditions \eqref{eqn:inequality_full} at different times $t$. This will become clear in the proof of Theorem~\ref{thm:optimal_deg_2}. At this moment we shall see that an argument very similar to the one above allows to derive from \eqref{eqn:2_jet_monotone} concrete relations between the objects appearing in the assertion of Theorem~\ref{thm:optimal_deg_2}. In particular, we get equations that every abnormal minimizing SR geodesic should satisfy. These were unknown to date.
    
    \begin{theorem}\label{thm:equations_of_motion}
    Within the assumptions and notation of Theorem~\ref{thm:optimal_deg_2}, consider the Pontryagin covector $\varphi(t)$ and a family of symmetric bi-linear maps $\PP(t)$ defined on a segment $[t_0',t_1']\ni t$. 
    
    Then we may  decompose $\PP(t)$ as $$\PP(t)[(\bb{1},\cc{1}),(\bb{1},\cc{1})]=\Pp(t)[\bb{1},\bb{1}]+\<\bb{1},\xi(t)>\cc{1}+a_2(t)(\cc{1})^2\ ,$$
       where $\Pp(t):\T_{q(t)}M\times\T_{q(t)}M\ra\R$ is a symmetric bi-linear map, $\xi(t):\T_{(q(t)}M\ra\R$ a linear map, and  $a_2(t)\in\R$ a number.\smallskip

        These objects satisfy the following conditions:
        \begin{itemize}
            \item maps $\Pp(t)$, $\xi(t)$ and $a_2(t)$ are of bounded variation, hence differentiable almost everywhere. 
            \item  for almost every $t\in[t_0',t_1']$  the quadratic  form  
        \begin{equation}
        \label{eqn:2_form_positive}\begin{split}
        \<\sum_{i=1}^k u_i(t) \Der^{2}_{q(t)} X_i[\bb{1},\bb{1}],\varphi(t)>-\dotPp(t)[\bb{1},\bb{1}]-2\Pp(t)\left[\sum_{i=1}^k u_i(t)\Der_{q(t)}X_i[\bb{1}],\bb{1}\right]+\\
        -\left(\<\bb{1},\dot\xi(t)>+\<\sum_{i=1}^k u_i(t)\Der_{q(t)} X_i[\bb{1}],\xi(t)>\right)\cdot \cc{1}-\dot{a_2}(t)\cdot (\cc{1})^2
        \end{split}\end{equation}
        is non-negative for every $(\bb{1},\cc{1})=(\bb{1}(t,\du),\cc{1}(t,\du))$.
        
        \item  for almost every $t\in[t_0',t_1']$ the following equations are satisfied for every  $\bb{1}=\bb{1}(t,\du)$  and  each $i=1,2,\hdots,k$
        \begin{align}
        \label{eqn:xi_u_1}
        \<X_i\big|_{q(t)},\xi(t)>+2\ a_2(t)\cdot u_i(t)=&0\\
        \label{eqn:Phi_b_1}
        2\<\Der_{q(t)}X_i[\bb{1}],\varphi(t)>+2\Pp(t)\left[X_i\big|_{q(t)},\bb{1}\right]+\<\bb{1},\xi(t)>\cdot u_i(t)=&0\ .
        \end{align} 
        \end{itemize}
    \end{theorem}
\begin{proof}
Since $\PP(t)$ is symmetric bilinear, it is obvious that it can be decomposed into parts $\Pp(t)$, $\xi(t)$ and $a_2(t)$ as stated in the assertion. Further these new objects must also be maps of bounded variation, since $\PP(t)$ was such.\smallskip

Now fix the control $\du$ and consider the map $$t\longmapsto F(t):=\<\bb{2}(t,\du),\varphi(t)>+\PP(t)[(\bb{1}(t,\du),\cc{1}(t,\du)),(\bb{1}(t,\du),\cc{1}(t,\du)\ .$$
It is non-decreasing by the assertion of Theorem~\ref{thm:optimal_deg_2}, and differentiable almost-everywhere, as $\PP(t)$ is of bounded variation and all the other objects appearing in the formula are at least absolutely continuous. Whenever defined $F(t)$ should have  non-negative derivative. We can  calculate this derivative using evolution equations \eqref{eqn:b1},\eqref{eqn:b2}, \eqref{eqn:c1_c2} and \eqref{eqn:pontryagin_covector} to arrive at
$$F'(t)=\text{(a part linear in $\du(t)$)}+\text{(a part quadratic in $\bb{1}(t,\du)$ and $\cc{1}(t,\du)$)}.$$
This should be non-negative for every possible $\du$, so the linear part must vanish. But this linear part is just (we skip simple but lengthy calculations)
\begin{align*}\sum_i \du_i(t)\bigg\{&\left(2\<\Der_{q(t)}X_i[\bb{1}(t,\du)],\varphi(t)>+2\Pp(t)\left[X_i\big|_{q(t)},\bb{1}(t,\du)\right]+\<\bb{1}(t,\du),\xi(t)>\cdot u_i(t)\right)+\\
&\cc{1}(t,\du)\left(\<X_i\big|_{q(t)},\xi(t)>+2\ a_2(t)\cdot u_i(t)\right)\bigg\}\ .
\end{align*}
It is easy to observe that the above expression is zero for every possible $\du(t)$ if and only if \eqref{eqn:xi_u_1}--\eqref{eqn:Phi_b_1} hold. 

Now we are left with $F'(t)=\text{a part quadratic in $\bb{1}(t,\du)$ and $\cc{1}(t,\du)$}$, which should be non-negative. This is precisely (again we skip simple calculations) condition \eqref{eqn:2_form_positive}
\end{proof}

\subsection{Discussion of the main result}\label{ssec:discussion}

Let us now discuss some particular consequences of Theorems~\ref{thm:optimal_deg_2} and \ref{thm:equations_of_motion}.

\paragraph{Problem of regularity} 
Observe that if $\xi(t)$ and/or $a_2(t)$ are non-zero (we shall call $q(t)$  a \emph{2-normal extremal} in such a situation -- see the next paragraph), then at least one of equations \eqref{eqn:xi_u_1}--\eqref{eqn:Phi_b_1} gives us an algebraic equation for the controls $u_i(t)$ (and thus an ODE for the trajectory) expressed by means of the data of Theorem~\ref{thm:optimal_deg_2}. \smallskip

Since all $a_2(t)$, $\xi(t)$, and $\Pp(t)$ are maps of bounded variation, and $X_i\big|_{q(t)}$ shares the same regularity as $q(t)$, we can only conclude that in such a situation $u_i(t)$'s are functions of bounded variation.  \textbf{If, however, the time-regularity of quadratic map $\PP(t)$ could be improved, we could also improve the regularity of the controls, and thus the SR geodesic.} Recall that maps $\PP(t)$ are defined as solutions of a certain optimization problem (Bellmann value function to be more precise -- we shall discuss that in detail in Subsection~\ref{ssec:char_ocp}). Therefore the question  whether 2-normal extremals are smooth boils down to an other regularity question in optimal control theory. I was personally unable to get any general results in this direction (yet some more specific ones are collected in Section~\ref{sec:suplementary}), but perhaps someone more experienced in optimal control theory could solve this problem.

\paragraph{2-normality and 2-abnormality}
Observe that if $a_2(t)\neq 0$, equation \eqref{eqn:xi_u_1} provides information about the controls $u_i(t)$. Basing on this fact we propose the following definition:

        \begin{definition}
        \label{def:2_normal_abnormal}
        Let $q(t)$, with $t\in [t_0,t_1]$ be a strictly abnormal sub-Riemannian trajectory such that the pair $(\varphi(t),\PP(t))$ described by Thm~\ref{thm:optimal_deg_2} exists for every $t\in(t_0,t_1]$.  If for every $t\in(t_0,t_1]$ we have $\PP(t)[(0,1),(0,1)]=a_2(t)\neq 0$ we call $q(t)$ a \emph{2-normal extremal}.
        Otherwise we say that  $q(t)$ is an \emph{2-abnormal extremal}. 
        \end{definition}

By the results of Thm~\ref{thm:optimal_deg_2} every strictly abnormal SR trajectory is either non-minimizing, or can be divided into a finite number of pieces, each of them being either 2-normal or 2-abnormal extremal.
\smallskip

The division of abnormal trajectories into the classes of 2-normal and 2-abnormal extremals proposed above shares some natural similarities with an analogous distinction between normal and abnormal trajectories in degree one. Namely, 2-normal extremals are characterized by a property that the second-order approximation of their extended end-point map can be strictly separated from the direction of the decreasing cost (see fig. on page \pageref{fig:normal_abnormal}). Further they satisfy an ODE (although we were not able to prove their regularity from this property). On the other hand, for 2-abnormal extremals this separation may not be strict, and equation \eqref{eqn:xi_u_1} gives us no information about the control $u(t)$ at the points where $a_2(t)=0$.    

Note that if $a_2(t)\equiv 0$, equations~\eqref{eqn:ev_xi_1} and \eqref{eqn:xi_u_1} are the same as equations for an abnormal extremal given at page \pageref{pontryagin_covector}. Thus we have two abnormal Pontryagin  covectors $\varphi(t)$ and $\xi(t)$ in that case (although the latter may be trivial).\smallskip

Finally, note that the classes of 2-normal and 2-abnormal extremals are, in general, not disjoint as an abnormal trajectory of co-rank greater than one may admit several distinct Pontryagin covectors.

\paragraph{Derivation of Goh conditions}
\label{par:Goh}
 Equations \eqref{eqn:xi_u_1}-\eqref{eqn:Phi_b_1} allow for an easy derivation of \emph{Goh conditions}. Put $\bb{1}=X_j$ in \eqref{eqn:Phi_b_1} to get
\begin{align*}
2\<\Der X_i[X_j],\varphi>\overset{\eqref{eqn:Phi_b_1}}=-2\Pp\left[ X_i,X_j\right]-\<X_j,\xi>u_i(t)\overset{\eqref{eqn:xi_u_1}}=-2\Pp\left[ X_i,X_j\right]+a_2\cdot u_j(t)u_i(t)\ .
\end{align*}	
The right-hand side of the above expression is symmetric in $i$ and $j$, hence 
 $$\<[X_i,X_j],\varphi>=\<\Der X_i[X_j],\varphi>-\<\Der X_j[X_i],\varphi>=0\ .$$
The Goh condition may be seen as a vanishing of a non-symmetric part of the derivatives \eqref{eqn:xi_u_1}--\eqref{eqn:Phi_b_1} due to the symmetric nature of the quadratic form $\PP(t)$.

\paragraph{Other observations}
 If $a_2(t)\equiv0$ (or more generally if $\dot a_2(t)=0$), then the bi-linear form \eqref{eqn:2_form_positive} contains a bare term linear in $\cc{1}$. This term should vanish for every possible choice of $\bb{1}$, since otherwise the form would not be positively-defined. This reasoning gives the evolution equation for $\xi(t)$:
	\begin{equation}
	\label{eqn:ev_xi_1}
\<\bb{1},\dot\xi(t)>+\<\sum_{i=1}^ku_i(t)\cdot \Der_{q(t)} X_i[\bb{1}],\xi(t)>=0\qquad \text{for every $\bb{1}\in \image\left(\Der_u \End{t_0,t}{q_0}\right)$,}
	\end{equation}
	nd so, $\xi(t)$ is a (possibly trivial) \emph{Pontryagin covector}.
\medskip

The question whether the bi-linear form \eqref{eqn:2_form_positive} is strictly positive or has some null-directions is closely related to the question whether the infimum value of the quadratic map $\PP(t)$ is actually achieved by some control. We will discuss it in more detail in Section~\ref{sec:suplementary}. At this point note that if the bi-linear form \eqref{eqn:2_form_positive} is constantly zero, and if the objects $\Pp(t)$, $\xi(t)$ and $a_2(t)$ are absolutely continuous w.r.t. $t$ then, in addition to the above equation \eqref{eqn:ev_xi_1}, we get $\dot{a_2}(t)\equiv 0$ and, at the end of the day, also evolution equation for $\Pp(2)$:
\begin{equation}
    \label{eqn:ev_Phi_2}
    \<\sum_{i=1}^k u_i(t) \Der^{2}_{q(t)} X_i[\bb{1},\bb{1}],\varphi>-\dotPp[\bb{1},\bb{1}]-2\Pp\left[\sum_{i=1}^k u_i(t)\Der_{q(t)}X_i[\bb{1}],\bb{1}\right]=0
\end{equation}
which should hold for every $\bb{1}\in \image\left(\Der_u \End{t_0,t}{q_0}\right)$. Informally speaking, these equations mean that the pair $(\varphi(t),\Pp(t))$ has  a natural evolution compatible with the flow of the control vector field $q\mapsto \sum_i u_i(t)X_i(q)$. For more details the Reader is directed to \cite{MJ_BS_adapted}.

\subsection{Examples}\label{ssec:examples}

\paragraph{Example - Martinet system}
Consider $M=\R^3\ni (x,y,z)$ and a distribution $\mathcal{D}$ spanned by orthonormal vector fields $X_1=\pa_x$, and $X_2=(1-x)\pa_y+\frac{x^2}2 \pa_z$. Let $q(t)=(0,t,0)$ be a trajectory corresponding to the control $(u_1(t)\equiv 0,u_2(t)\equiv 1)$. It is well-known that $q(t)$ is an abnormal minimizing geodesic, and that $\phi(t)=\dd z$ is the related adjoint curve \cite{Montgomery_2006}. We want to see how Thm.~\ref{thm:optimal_deg_2} works in this case. 
\smallskip

First we would like to derive equations \eqref{eqn:b1} and \eqref{eqn:b2} for the considered situation. Let us decompose 
$$\bb{1}(t)=\bb{1}_x(t)\cdot \pa_x+\bb{1}_y(t)\cdot \pa_y+\bb{1}_z(t)\cdot \pa_z\quad\text{and}\quad \bb{2}(t)=\bb{2}_x(t)\cdot \pa_x+\bb{2}_y(t)\cdot \pa_y+\bb{2}_z(t)\cdot\pa_z\ .$$
Now note that 
$$\sum_{i=1}^2 u_i(t)\Der_{q(t)} X_i[\bb{1}]=-\dd x[\bb{1}]\cdot \pa_y\quad\text{and}\quad \sum_{i=1}^2 u_i(t)\Der^2_{q(t)}X_i[ \bb{1},\bb{1}]=\dd x\otimes\dd x[\bb{1},\bb{1}]\cdot \pa_z\ .$$ 

From the above observation equation \eqref{eqn:b1} reads as
$$\dot{\bb{1}_x}(t)=\du_1(t),\quad \dot{\bb{1}_y}(t)=-\bb{1}_x(t)+\du_2(t),\quad \dot{\bb{1}_z}(t)=0,\quad\text{and}\quad \dot{\cc{1}}(t)=\du_2(t)\ .$$
In particular,
\begin{equation}
\label{eqn:martinet}
\bb{1}_y(t)=-\int_0^t \bb{1}_x(\tau)\ \dd \tau +\int_0^t \du_2(\tau)\ \dd \tau=-\int_0^t \bb{1}_x(\tau)\ \dd \tau +\cc{1}(t)\ .
\end{equation}
Similarly, \eqref{eqn:b2} gives us
$$\dot{\bb{2}_x}(t)=0,\quad \dot{\bb{2}_y}(t)=-\bb{2}_x(t)-\du_2(t)\bb{1}_x(t),\quad \dot{\bb{2}_z}(t)=\bb{1}_x(t)^2,\quad\text{and}\quad \dot{\cc{2}}(t)=\du_1(t)^2+\du_2(t)^2\ .$$
In consequence, by the Cauchy-Schwarz inequality we can estimate
\begin{align*}
\<\bb{2}(t),\phi(t)>=\bb{2}_z(t)=\int_0^t \bb{1}_x(\tau)^2\ \dd\tau\overset{\text{C-S}}\geq\frac 1t\cdot \left(\int_0^t \bb{1}_x(\tau)\ \dd\tau\right)^2= \frac 1t\cdot \left[\cc{1}(t)-\bb{1}_y(t)\right]^2\ .
\end{align*}
It is easy to see that the above inequality is always strict,  yet on the other hand, it can be satisfied with an arbitrary accuracy.  It follows that every minimal 2-jet is finite, but the minimum is never attained. We conclude that $\costmin{0,t}{\bb{1},\cc{1}}=\frac 1t\cdot \left[\cc{1}(t)-\bb{1}_y(t)\right]^2$ and, consequently,
$$\Pp(t)=\frac 1t \dd y\otimes\dd y,\quad \xi(t)=-\frac 2t \dd y\quad\text{and}\quad a_2(t)=\frac 1t\ .$$
Since $a_2(t)\neq 0$ the considered trajectory is a \emph{2-normal extremal} in the sense of Definition~\ref{def:2_normal_abnormal}. We leave to the Reader checking that conditions  \eqref{eqn:xi_u_1}-\eqref{eqn:Phi_b_1} hold for the above data. Further, condition \eqref{eqn:2_form_positive} reads as
$$\left(\bb{1}_x(t)+\frac 1t \bb{1}_y(t)\right)^2-\left(\frac 2{t^2}\bb{1}_y(t) +\frac 2 t\bb{1}_x(t)\right)\cc{1}(t)+\frac 1{t^2}\cc{1}(t)^2=\left(\bb{1}_x(t)+\frac 1t \bb{1}_y(t)-\frac 1t \cc{1}(t)\right)^2\geq 0$$ 
and of course holds.

\paragraph{Example -- a corner on a free Carnot group}  The example below was studied by Hakavouri and Le Donne as communicated to us by \cite{ELD_new}. It turns out to be a 2-abnormal extremal.  \smallskip

Consider a free Carnot group $G$ or rank 2 and step 4. That is, the Lie algebra of left-invariant vector fields on $G$ is generated by two fields $X_1$, $X_2$, whose Lie brackets of order higher than 4 vanish, and Lie  brackets of order less or equal 4 have no relations apart from the obvious ones arising either from the skew-symmetry, or from the Jacobi identity. It follows that all left-invariant vector fields on $G$ are spanned (over $\R$) by fields: $X_1$, $X_2$,
$Y:=[X_1,X_2]$, $Z_1:=[X_1,Y]=[X_1,[X_1,X_2]]$, $Z_2:=[X_2,-Y]=[X_2,[X_2.X_1]]$, $W_1:=[X_1,Z_1]=[X_1,[X_1,[X_1,X_2]]]$, $W_2:=[X_2,Z_2]=[X_2,[X_2,[X_2,X_1]]]$, and $W_3:=[X_2,Z_1]=[X_1,-Z_2]=[X_1,[X_2,[X_1,X_2]]]$.Let us consider a left-invariant SR geodesic problem on $G$, constituted by setting $\D_q=\operatorname{span}_{\R}\{X_1,X_2\}\big|_q$ and the assumption that the fields $X_1$ and $X_2$ are orthonormal. \smallskip

In such a setting let $q:[-1,1]\ra G$ be the \emph{corner} controlled by $u(t)=(\chi_{[-1,0]}(t),\chi_{[0,1]}(t))$, i.e.
$$q(t)=\begin{cases}\exp(t\cdot X_1) & \text{for $t\in[-1,0]$}\\
\exp(t\cdot X_2) & \text{for $t\in[0,1]$}
\end{cases}\ .$$
It turns out that $q(t)$ is a strictly abnormal SR extremal, with 
$$\image \Der_u \END{-1,1}{}=\operatorname{span}_{\R}\{X_1,X_2,Y,Z_1,Z_2,W_1,W_2\}\big|_{q(1)}\oplus\R\ .$$ 
In consequence, there is only one (up to rescaling) candidate for the Pontryagin covector $\varphi(t)\in \T^\ast_{q(t)}G$ uniquely defined by conditions  $$\<W_3\big|_{q(t)},\varphi(t)>=1,\ \text{and}\ \<\cdot, \varphi(t)>=0\ \text{on the space}\ \operatorname{span}_{\R}\{X_1,X_2,Y,Z_1,Z_2,W_1,W_2\}\big|_{q(t)}\subset\T_{q(t)}G\ .$$

In this particular situation we were able to calculate $\<\bbb{2}(1,\du),\varphi(1)>$ for every control $\du=(\du_1,\du_2)\in L^\infty ([-1,1],\R^2)$. We skip the details of this computation, as it is rather long and  would require introduction of some additional theory regarding invariant SR structures on Lie groups. To state our result we need to introduce the function  $f:G\ra\R$ defined by 
$$f\left(\exp(w_3W_3)\exp(w_2 W_2)\exp (w_1 W_1)\exp (z_2 Z_2)\exp (z_1 Z_1)\exp (y Y)\exp(x_2  X_2)\exp(x_1 X_1)q(1)\right):=w_3\ $$
(obviously $\Der_{q(1)} f=\varphi(1)$), and a covector $\rho\in \T^\ast_{q(1)}G$  defined by conditions 
$$\<X_1\big|_{q(1)},\rho>=1,\ \text{and}\ \<\cdot, \rho>=0\ \text{on the space}\ \operatorname{span}_{\R}\{X_2,Y,Z_1,Z_2,W_1,W_2,W_3\}\big|_{q(1)}\ ,$$ 
i.e. $\<\bb{1}(1,\du),\rho>$ is the $X_1$-component of the vector $\bb{1}(1,\du)\in\T_{q(1)}G$. \smallskip

It turns out that for every $\du\in L^\infty([-1,1],\R^2)$ the pair $(\bbb{1}(1,\du), \bbb{2}(1,\du))$ is subject to the following equality
\begin{align*}
 \<\bbb{2}(1,\du),\varphi(1)>+\Der^2_{q(1)} f[\bb{1}(1,\du),\bb{1}(1,\du)]+\frac 12 \<\bb{1}(1,\du),\rho>^2=\\
 \int_0^1 t\left(\int_{-1}^t\du_1(s)\ \dd s\right)^2\ \dd t  - \int_{-1}^0 t\left(\int_{-1}^t\du_2(s)\ \dd s\right)^2\ \dd t\ .
\end{align*}
Clearly the right-hand side is positively defined. We conclude that in the considered case inequality \eqref{eqn:inequality_full} holds with $\PP(1)[(\bb{1},\cc{1}),(\bb{1},\cc{1})]=-\Der^2_{q(1)}f[\bb{1},\bb{1}]-\frac 12 \<\bb{1},\rho>^2$. As the the dependence on $\cc{1}$ is trivial, $a_2(1)=0$, and  hence we deal with a \emph{2-abnormal extremal} in the sense of Definition~\ref{def:2_normal_abnormal}. It is worth to mention that, by the results of \cite{Hakavuori_Donne_2018}, a corner curve cannot be a SR geodesic.

\section{Proof of Theorem~\ref{thm:optimal_deg_2}}
\label{sec:proof}

\subsection{The characteristic optimal control problem}\label{ssec:char_ocp}

Let us briefly sketch the content of this subsection. We would like to see Theorem~\ref{thm:optimal_deg_2} as an improvement on Lemma~\ref{lem:AS_index_zero}. If so, then the new part is the presence of the quadratic map $\PP(t)$ related with minimizing the value of $\<\bb{2}(t,\du),\varphi(t)>$ while fixing the values of $\bb{1}(t,\du)$ and $\cc{1}(t,\du)$. 

We shall formalize this minimization procedure by considering the value function $\costmin{\tau}{\cdot}$ of a certain optimal control problem (the \emph{characteristic control problem} of the SR trajectory). Giving a precise definition of this problem, as well as of underlying \emph{characteristic control system}, will conclude  Subsection~\ref{ssec:char_ocp}. It is worth to mention, that for technical reasons we will prefer to study curves $(\bb{1}(t,\du),\bb{2}(t,\du))$ in a special coordinate system -- \emph{adapted coordinates} -- introduced in \cite{MJ_BS_adapted}. 
After introducing thea above notions we will sketch the strategy of the proof at the beginning of the next Subsection~\ref{ssec:prop_min_2_jet}.

\paragraph{Technical preparations (adapted coordinates)}
In the proof of Theorem~\ref{thm:optimal_deg_2} we are, in particular, interested in the geometry of the set of pairs $(\bb{1}(t,\du),\bb{2}(t,\du))$, which represent the 2-jets of $s\mapsto \End{t_0,t_1}{}[u+s\cdot \du]$ at $s=0$. These objects are described by the system of ODEs \eqref{eqn:b1}--\eqref{eqn:b2} which, for us, would be convenient  to treat as a control system, steered by $\du\in \LL{t_0}{t_1}$. There are, however, two issues with such an approach. First of all, equations \eqref{eqn:b1}--\eqref{eqn:b2} are quite complicated, depending not only linearly on $\du$, but also on  the basic control $u$ by formulas that involve derivatives of vector fields $X_i$. Secondly, for each $t$ the pair $(\bb{1}(t,\du),\bb{2}(t,\du))$ belongs to a different space --  $\T_{q(t)}^2M$. It turns out that both complications can  be solved simultaneously by a suitable change of coordinates.  It will be given by a family of maps $\Psi^{(2)}(t):\T^2_{q(t)}M\lra \R^n\times\R^n$ constructed as follows \cite{MJ_BS_adapted}.

Let $u\in\LL{t_0}{t}$ be a control, $q(t)$ the related trajectory of \eqref{eqn:trajectory_u}, and choose a basis $(\psi_a)_{a=1.2,\hdots,n}$ of the cotangent space $\T^\ast_{q_0}M$. First define a 1-parameter family of linear maps $\lin{1}(t):\T_{q(t)}M\ra \R^n$ by setting
\begin{equation}\label{eqn:ev_lin_1}
    \begin{split}
&\dotlin{1}(t)[b]+\sum_i u_i(t)\cdot\lin{1}(t)\left[ \Der_{q(t)} X_i[b]\right]=0\qquad \text{for every $b\in\T_{q(t)}M$}\\
&\lin{1}(t_0)=(\psi_1,\hdots,\psi_n),
    \end{split}
\end{equation}
and a 1-parameter family of symmetric bi-linear maps $\lin{2}(t):\T_{q(t)}M\times \T_{q(t)}M\ra \R^n$ by setting
\begin{align*}
&\dotlin{2}(t)[b,b]+\sum_i u_i(t)\cdot\left\{\lin{1}(t)\left[\Der^2_{q(t)}X_i[b,b]\right] +2\lin{2}(t)\left[\Der_{q(t)} X_i[b],b\right]\right\}=0\qquad \text{for every $b\in\T_{q(t)}M$}\\
&\lin{2}(t_0)=(0,\hdots,0)\ .
\end{align*}
Now the map $\Psi^{(2)}(t)$ is defined as follows:
$$\Psi^{(2)}(t)(\bb{1},\bb{2})=\left(\lin{1}(t)\left[\bb{1}\right], \lin{1}(t)\left[\bb{2}\right]+\frac 1{2!}\lin{2}(t)[\bb{1},\bb{1}] \right)\ ,$$
where $(\bb{1},\bb{2})\in\T^2_{q(t)}M$. After \cite{MJ_BS_adapted} we call $\Psi^{(2)}(t)$ the \emph{transformation of adapted coordinates}. The name ``adapted'' means ``adapted to the geometry of the evolution equation \eqref{eqn:trajectory_u}'' in the sense that $\Psi^{(2)}(t)$ acts on 2-jets in the same way as the flow of the control vector field $X_u(t,q)=\sum_i u_i(t)X_i(q)$.  In particular, this manifests itself by the fact that passing to adapted coordinates simplifies the form of equations \eqref{eqn:b1}--\eqref{eqn:b2}. 

\begin{lemma}[\cite{MJ_BS_adapted}]
\label{lem:adapted}
For every $t\in[t_0,t_1]$ the map $\lin{1}(t)$ is an isomorphism between the fibre $\T_{q(t)}M$ and $\R^n$. By the same token, for every $t\in[t_0,t_1]$ the map $\Psi^{(2)}(t)$ is an isomorphism between the fibre $\T^2_{q(t)}M$ and $\R^n\times\R^n$. 
\smallskip

Given $\du\in\LL{t_0}{t_1}$ set
\begin{equation}\label{eqn:def_q12_du}\left(\q{1}(t,\du),\q{2}(t,\du)\right):=\Psi^{(2)}(t)\left(\bb{1}(t,\du),\bb{2}(t,\du)\right)\ .
\end{equation}
Then the curve $(\q{1}(t,\du),\q{2}(t,\du))$ is subject to the following evolution equations
\begin{align}
\label{eqn:ev_q1}
\dotq{1}(t,\du)&=\sum_i \du_i(t)\cdot  \Y{i}{1}(t)\\
\label{eqn:ev_q2}
\dotq{2}(t,\du)&=\sum_i \du_i(t)\cdot \Y{i}{2}(t,\q{1}(t,\du)) 
\ ,
\end{align}
where, after denoting the inverse of $\lin{1}(t)$ by $A(t)$, we have
$$\Y{i}{1}(t):=\lin{1}(t)\left[X_i\big|_{q(t)}\right]\quad\text{and}\quad   \Y{i}{2}(t,\q{1}) :=\lin{1}(t)\left[\Der_{q(t)}X_i[A(t) \q{1}]\right]+\lin{2}(t)\left[X_i, A(t)q^{(1)}\right]\ .$$
\medskip 

 Finally, if the control $u$ is of class $C^r$, with $r=0_-,1_-,1,2,3,\hdots$, then $\lin{1}(t)$, the inverse of $\lin{1}(t)$,  and $\lin{2}(t)$ are of class $C^{r+1}$ with respect to $t$.  

\end{lemma}
The above result can be checked by a direct calculation. The $C^{r+1}$-regularity follows easily from the evolution equations for $\lin{1}(t)$ and $\lin{2}(t)$. For details and discussion, we refer to the original publication \cite{MJ_BS_adapted}. 


\paragraph{The characteristic control system and its reachable sets.}
Since maps $\lin{1}(t)$ constructed in Lemma~\ref{lem:adapted} are reversible, a pair $(\q{1}(t,\du),\cc{1}(t,\du))$ contains the same information as a pair $(\bb{1}(t,\du),\cc{1}(t,\du))$, describing the first term of the 2-jet expansion of the extended end-point map $\END{t_0,t}{q_0}$. Thanks to the simple form of equation \eqref{eqn:ev_q1} the evolution of the former pair is, however, much easier to analyse. Therefore to understand the geometry of $\END{t_0,t}{q_0}$, we shall now study the pairs $(\q{1}(t,\du),\cc{1}(t,\du))\in\R^n\times\R$. In accordance with Rem.~\ref{rem:notation} denote $\qq{1}(t,\du):=(\q{1}(t,\du),\cc{1}(t,\du))$ and $\zero:=(0_n,0)\in\R^n\times\R$. 

\begin{definition}
\label{def:char_CS}
By a \emph{characteristic control system} of the trajectory $q(t)$, with $t\in[t_0,t_1]$, we will understand the following control system in $\R^{n+1}\ni \qq{1}=(\q{1},\cc{1})$:
\begin{equation}
\label{eqn:CS_deg_1}
\tag{$\Lambda_1$}
    \begin{split}
        \dotq{1}(t)&=\sum_i \du_i(t)\cdot  \Y{i}{1}(t)\\
        \dot {\cc{1}}(t)&=\sum_{i=1}^k \du_i(t)\cdot u_i(t)
    \end{split}
\end{equation}
Here $\Y{i}{1}(t)=\lin{1}(t)\left[X_i\big|_{q(t)}\right]$, and we treat $\du \in \LL{t_0}{t_1}$ as a control. 
\smallskip

By $\RR(t_0,\tau)$ we shall denote \emph{time-$\tau$ reachable set of the system \eqref{eqn:CS_deg_1}} for the initial conditions $\qq{1}(t_0)=\zero$. That is
\begin{align*}\RR(t_0,\tau):=\{\qq{1}(\tau)\that &\text{$\qq{1}(t)$ satisfy \eqref{eqn:CS_deg_1}}\text{ for some $\du\in \LL{t_0}{\tau}$ with $\qq{1}(t_0)=\zero$.} \}
\end{align*}
\smallskip

We will need also another notion related with the characteristic control system. Choose $\tau\in [t_0,t_1]$. We will say that a vector $\vv\in\R^{n+1}$ is \emph{controllable to zero from time $\tau$} if there exists a trajectory $\qq{1}(t)$ of the control system \eqref{eqn:CS_deg_1} with the initial condition $\qq{1}(\tau)=\vv$ such that $\qq{1}(t_1)=\zero$. 
\end{definition}

\paragraph{Pontryagin covector in adapted coordinates}
Observe that equation \eqref{eqn:ev_lin_1} defining the map $\lin{1}(t)$ looks exactly the same, as the evolution equation of the Pontryagin covector \eqref{eqn:pontryagin_covector}. Therefore it is easy to translate  conditions of being a Pontryagin covector to the setting of adapted coordinates. 

\begin{proposition}\label{prop:pontryagin_in_adapted}
A curve $\varphi(t)\in \T_{q(t)} M$ is a Pontryagin covector (as defined by conditions \eqref{eqn:pontryagin_covector}) if and only if there exists a covector $\psi_0\in (\R^n)^\ast$ such that 
\begin{align*}&\varphi(t)=\lin{1}(t)^\ast\psi_0\quad\text{and}\\
&\<\Y{i}{1}(t),\psi_0>=0\quad\text{for every $i=1,2,\hdots,k$ and $t\in[t_0,t_1]$.} 
\end{align*}
\smallskip

Further if $\varphi(t)$ is as above and if the pairs $(\bb{1}(t,\du),\bb{2}(t,\du)$ and $(\q{1}(t,\du),\q{2}(t,\du))$ are related by \eqref{eqn:def_q12_du} we have 
\begin{align*}
   \<\bb{1}(t,\du),\varphi(t)>=&\<\q{1}(t,\du),\psi_0>\quad\text{and}\\
   \<\bb{2}(t,\du),\varphi(t)>=&\<\q{2}(t,\du),\psi_0>+\frac 12\<\lin{2}(t)[\bb{1}(t,\du),\bb{1}(t,\du)],\lin{1}(t)^\ast \psi_0>=\\
   &\<\q{2}(t,\du),\psi_0>+\frac 12\<A(t)\lin{2}(t)[A(t)\q{1}(t,\du),A(t)\q{1}(t,\du)],\psi_0>\ ,
\end{align*}
where $A(t)$ denotes the inverse of $\lin{1}(t)$.
\end{proposition}

Recall that in  Lemmas~\ref{lem:AS_index} and \ref{lem:AS_index_zero} we study a pairing $\<\bb{2}(t,\du),\varphi(t)>$ for a Pontryagin covector $\varphi(t)$.

 As we see, when restricting our attention to $\du\in \ker\END{t_0,t}{}$ (as is the case in the discussed Lemmas) the pairing $\<\bb{2}(t,\du),\varphi(t)>$ equals to $\<\q{2}(t,\du),\psi_0>$. Without such a restriction both pairings are equal up to a term quadratic in $\q{1}(t,\du)$.  

\paragraph{The characteristic optimal control problem}
Motivated by Lemmas~\ref{lem:AS_index} and \ref{lem:AS_index_zero} and our observations from the previous paragraph we define the following optimization problem for trajectories of the characteristic control system \eqref{eqn:CS_deg_1}.  

        \begin{definition}\label{def:char_cs} Let $q(t)$, with $t\in [t_0,t_1]$, be a SR trajectory corresponding to the control $u\in \LL{t_0}{t_1}$. Choose $\psi_0\in(\R^{n})^\ast$, $\tau\in[t_0,t_1]$ and $\vv\in\RR(t_0,\tau)$. By a \emph{characteristic optimal control problem} (\emph{characteristic OCP}, in short) of the trajectory $q(t)$ we understand a problem of  minimizing the functional 
        $$\cost{t_0,\tau}{\du}:=\<\q{2}(\tau,\du),\psi_0>$$
        in the class of  controls $\du\in \LL{t_0}{\tau}$ such that $\qq{1}(t_0,\du)=\zero$ and $\qq{1}(\tau,\du)=\vv$. (In other words, $\du$ steers \eqref{eqn:CS_deg_1} from $\zero$ to $\vv$ on $[t_0,\tau]$.) \smallskip
         Note that  by \eqref{eqn:ev_q2}
            $$\cost{t_0,\tau}{\du}=\int_{t_0}^\tau \sum_i \du_i(t)\cdot \<\Y{i}{2}(t,\q{1}(t,\du)),\psi_0 >\ \dd t
            \ , $$
        hence we are actually solving a Larange-type optimal control problem.\medskip
        
        \end{definition}

Putting for the moment the question of the existence of solutions aside, define \emph{the minimal 2-jet} to be the Bellman value function related with this problem 
$$\costmin{t_0,\tau}{\vv}:=\inf \{\cost{t_0,\tau}{\du}\that \text{$\du\in\LL{t_0}{\tau}$ is a control steering \eqref{eqn:CS_deg_1} from $\zero$ to $\vv$ on  $[t_0,\tau]$} \}\ .$$

We admit a possibility that $\costmin{t_0,\tau}{\vv}$ 
equals to $-\infty$.  \medskip

Finally, note that the cost functional $\cost{t_0,\tau}{\cdot}$ is  quadratic, in the sense that $\cost{t_0,\tau}{\du}=\blincost{t_0,\tau}{\du}{\du}$, where $\blincost{\tau}{\cdot}{\cdot}:\LL{t_0}{t_1}\times\LL{t_0}{t_1}\lra \R$ is a symmetric bilinear map defined by 
\begin{align*}2\ \blincost{\tau}{\du}{\dv}:=
\int_{t_0}^\tau \sum_i  \<\du_i(t)\cdot \Y{i}{2}(t,\q{1}(t,\dv))+\dv_i(t)\cdot \Y{i}{2}(t,\q{1}(t,\du)),\psi_0 >\ \dd t \ .
\end{align*}

\subsection{Properties of the minimal 2-jet}\label{ssec:prop_min_2_jet}

\paragraph{Strategy of the proof}

In the remaining part of this section, we are going to investigate properties of the minimal 2-jet $\costmin{t_0,\tau}{\cdot}$ under the assumption that $q(t)$ is an abnormal minimizing SR geodesic. A crucial part in the proof will be showing Lemma~\ref{lem:final_division} stating that $\costmin{t_0,\tau}{\vv}$ is finite for all possible $\tau$'s and $\vv$'s, provided that the underlying SR trajectory is suitably divided into pieces. This will be done in several steps as listed below:
\begin{itemize}
    \item First, we will prove Lemma~\ref{lem:AS_index_zero_1} and the related Corollary~\ref{cor:AS_index_zero_1} stating that for an abnormal minimizing SR geodesic the value function $\costmin{t_0,\tau}{\zero}$ is finite after a suitable division of the trajectory. This is basically a reformulation of Lemma~\ref{lem:AS_index_zero}. 
    \item Secondly, in Lemma~\ref{lem:fundamental} and Corollary~\ref{cor:fundamental}, we show that if an element $\vv$ is controllable to zero from time $\tau$ then $\costmin{t_0,\tau}{\vv}$ is finite
    \item Lastly, in Lemma~\ref{lem:control_to_zero} we show that, after a suitable division of the trajectory, every element $\vv$ of the reachable set of \eqref{eqn:CS_deg_1} is controllable to zero.
\end{itemize} 
With all the above points to conclude the proof it is enough to prove Theorem~\ref{thm:properties_of_Q} which states that if $\costmin{t_0,\tau}{\vv}$ is finite fpr every possible $\vv$, then it is quadratic in $\vv$. 

\paragraph{Step 1 -- setting the AS index to zero}

To simplify the notation let us denote the kernel of the first derivative of the extended end-point map $\END{t_0,t}{}$ at $u$ by
$$\U{0}(t_0,t):=\ker \Der_u\END{t_0,t}{}=\{\du\in\LL{t_0}{t_1}\that \bb{1}(t,\du)=0, \cc{1}(t,\du)=0\}\ .$$
Recall that in Lemma~\ref{lem:AS_index} we study quadratic maps 
$$\G{t_0,t}:\U{0}(t_0,t)\lra\R\qquad \G{t_0,t}:\du\longmapsto \<\bb{2}(t,\du),\varphi(t)>\overset{\text{Prop.~\ref{prop:pontryagin_in_adapted}}}=\<\q{2}(t,\du),\psi_0>\ ,$$
where $\varphi(t)$ is a certain Pontryagin covector. In the last equality we use the notation from Proposition~\ref{prop:pontryagin_in_adapted}. Given a covector $\psi_0$ as above we may repeat the above  construction on subintervals  $[t_0',t_1']\subset[t_0,t_1]$ to produce analogous objects: $\U{0}(t_0',t_1')$ and  $\G{t_0',t_1'}$, respectively. Actually, the extension of an $L^\infty$-map $\du:[t_0',t_1']\ra\R^k$ by zero on $[t_0,t_0')\cup(t_1',t_1]$ gives us a canonical inclusion $\LL{t_0'}{t_1'}\subset\LL{t_0}{t_1}$, which allows to interpret each $\U{0}(t_0',t_1')$ as a subset of $\U{0}(t_0,t_1)$ and, consequently, $\G{t_0',t_1'}$ as a restriction of $\G{t_0,t_1}$ to $\U{0}(t_0',t_1')$. With this simple observation the below result is now easy to prove. It is an obvious reformulation of Lemma~\ref{lem:AS_index_zero}.

        \begin{lemma}\label{lem:AS_index_zero_1} Under the assumptions of Lemma~\ref{lem:AS_index}, there exist a non-zero covector $\psi_0\in(\R^n)^\ast$, and at most $r+1$ points $\tau_0=t_0\leq \tau_1\leq\tau_2\leq\hdots\leq \tau_{r'}=t_1$ such that for each subinterval $[t_0',t_1']\subset(\tau_j,\tau_{j+1}]$, the quadratic map
        $$\G{t_0',t_1'}:\U{0}(t_0',t_1')\lra\R$$
        is non-negatively defined. 
        \end{lemma}
\begin{proof} We take a Pontryagin covector $\varphi(t)$ as in the assertion of Lemma~\ref{lem:AS_index} and let $\psi_0$ be the corresponding element of $(\R^n)^\ast$ given by Proposition~\ref{prop:pontryagin_in_adapted}. Using this element we define the map $\G{t_0,t}$ as above. 

Since for $t<t'$ we have a canonical inclusion $\U{0}(t_0,t)\subset \U{0}(t_0,t')$, the integer-valued function 
$$I:[t_0,t_1]\ni t\longmapsto \text{the negative index of $\G{t_0,t_1}$ on $\U{0}(t_0,t)$} $$
is non-decreasing. Further, as being negatively defined on a subspace is an open property, $I$ is left semi-continuous. By Lemma~\ref{lem:AS_index} the map $I$ has at most $r-1$ points of growth. Together with $t_0$ and $t_1$ this gives us at most $r+1$ points $\tau_0=t_0\leq \tau_1\leq\tau_2\leq\hdots\leq \tau_{r'}=t_1$, with the property that $I$ is constant on $(\tau_j,\tau_{j+1}]$ Take now any $[t_0',t_1']\subset (\tau_j,\tau_{j+1}]$ and assume that $\G{t_0,t_1}$ is not non-negatively defined on $\U{0}(t_0',t_1')$. For $\tau_j+\eps<t_0'$ the sets $\U{0}(t_0,\tau_j+\eps)$ and $\U{0}(t_0',t_1')$ are $\G{t_0,t_1}$-orthogonal (consisting of controls with disjoints supports, for which $\q{1}$ vanishes at the end-points), it follows that the negative index of $\G{t_0,t_1}$ on $\U{0}(t_0,\tau_j+\eps)\cup\U{0}(t_0',t_1')\subset\U{0}(t_0,t_1')$ is greater than the negative index of $\G{t_0,t_1}$ on $\U{0}(t_0,\tau_j+\eps)$. This is, however, impossible as the index function $I$ had no growth-points between $\tau_j+\eps$ and $\tau_{j+1}\geq t_1'$.
\end{proof}

From the above we easily get
\begin{corollary}\label{cor:AS_index_zero_1} Under the assumptions and notation of Lemma~\ref{lem:AS_index_zero_1}, consider the restriction of an abnormal minimizing SR geodesic $q(t)$ to a subinterval $[t_0',t_1']\subset(\tau_j,\tau_{j+1}]$. Then the related value function $\costmin{t_0',t}{\zero}=0$ for every $t\in[t_0',t_1']$.
\end{corollary}

In this context it is convenient to state the following simple fact.

\begin{proposition}
        \label{fact:Q_zero_is_zero}
        If $\costmin{t_0,t}{\zero}$ is finite then it equals to $0$.
        \end{proposition}
        \begin{proof}
        Of course taking a trivial control $\du\equiv 0$ gives us $\cost{t_0,t}{\du}=0$, and hence $\costmin{t_0,t}{\zero}\leq 0$. Assume that there exists some $\dv\in \U{0}(t_0,t)$ such that $\cost{t_0,t}{\dv}=-c<0$. Note that, since $\U{0}(t_0,t)$ is a linear space, and  since the functional $\cost{t_0,t}{\cdot}$ is quadratic, for every $\lambda\in \R$ we have $\lambda\cdot\dv\in \U{0}(t_0,t)$ and $\cost{t_0,t}{\lambda\cdot\dv}=-\lambda^2 c$. Taking $\lambda\to \infty$ gives us $\costmin{t_0,t}{\zero}=-\infty$. We get a contradiction with the assumption that $\costmin{t_0,t}{\zero}$ was finite. 
        \end{proof}

\paragraph{Step 2 -- Bellman-like inequalities} The minimal 2-jet $\costmin{t_0,t}{\cdot}$ is defined as the Bellman's value function for an optimal control problem. Thus it should not be surprising that it shares its basic properties.
Namely, it turns out, that for a fixed control $\du\in \LL{t_0}{t_1}$, the distance between $\cost{t_0,t}{\du}$ and the infimum value $\costmin{t_0,t}{\vv=\qq{1}(t,\du)}$  is a non-decreasing function of time $t$. This property  will be of crucial importance later. 

        \begin{lemma}[Bellman's inequality]
        \label{lem:fundamental}
        For every $\du\in \LL{t_0}{t_1}$, function 
        $$t\longmapsto \cost{t_0,t}{\du}-\costmin{t_0,t}{\qq{1}(t,\du)}$$
        is non-decreasing. 
        \end{lemma}
\begin{proof}
Fix a control $\du\in\LL{t_0}{t_1}$, and assume for simplicity that $\cost{t_0,t}{\du}-\costmin{t_0,t}{\qq{1}(t,\du)}=5$, while $\cost{t_0,t'}{\du}-\costmin{t_0,t'}{\qq{1}(t',\du)}=1$ for some $t_0\leq t<t'\leq t_1$. Since $\costmin{t_0,t}{\cdot}$ is defined as the infimum, there is some control $\du'\in\LL{t_0}{t}$ such that $\qq{1}(t,\du)=\qq{1}(t,\du')$, while $\cost{t_0,t}{\du'}-\costmin{t_0,t}{\qq{1}(t,\du)}<5-1=4$. Thus
\begin{equation}
    \label{eqn:inequality}
    \cost{t_0,t}{\du'}-\cost{t_0,t}{\du}<-1\ .
\end{equation}
Now consider a concatenated control $\widetilde{\du}:=\du'|_{[t_0,t]}\circ\du|_{[t,t']}\in\LL{t_0}{t'}$. Obviously, $\qq{1}(t,\widetilde{\du})=\qq{1}(t,\du')=\qq{1}(t,\du)$; and $\qq{1}(t',\widetilde{\du})=\qq{1}(t',\du)$ as both controls $\widetilde{\du}$ and $\du$ equal on $[t,t']$ and share the same initial point $\qq{1}(t,\du)$. On the other hand,
\begin{align*}
\cost{t_0,t'}{\widetilde{\du}}=&\cost{t_0,t}{\du'}+\int_t^{t'}\sum \du_i(\tau)\cdot \<\Y{i}{2}(\tau,\q{1}(\tau,\du)),\psi_0>\ \dd\tau=\\
&\cost{t_0,t}{\du'}+\cost{t_0,t'}{\du}-\cost{t_0,t}{\du}\overset{\eqref{eqn:inequality}}<\cost{t_0,t'}{\du}-1\ .
\end{align*}
It follows that 
\begin{align*}\cost{t_0,t'}{\widetilde{\du}}-\costmin{t_0,t'}{\qq{1}(t',\widetilde{\du})}=&\cost{t_0,t'}{\widetilde{\du}}-\costmin{t_0,t'}{\qq{1}(t',\du)}\\
<&\cost{t_0,t'}{\du}-\costmin{t_0,t'}{\qq{1}(t',\du)}-1=0\ .
\end{align*}
This is a contradiction, with the definition of $\costmin{t_0,t'}{\cdot}$ as the infimum. 
\end{proof}
Combining the above with the results from the previous step of the proof we get a result that states that if we can steer an element $\vv\in\R^{n+1}$ to zero, then the minimal 2-jet is finite at this point. 

\begin{corollary}\label{cor:fundamental}Under the assumptions and notation of Lemma~\ref{lem:AS_index_zero_1}, consider the restriction of an abnormal minimizing SR geodesic $q(t)$ to a subinterval $[t_0',t_1']\subset(\tau_j,\tau_{j+1}]$.  Let $\vv\in\RR(t_0',\tau)\subset \R^{n+1}$ be a vector controllable to zero from time $\tau\in[t_0',t_1']$ (in the sense of Definition~\ref{def:char_CS}). Then the minimal 2-jet $\costmin{t_0',\tau}{\vv}$ is finite.
\end{corollary}
\begin{proof}
Let $\du_1\in \LL{t_0'}{\tau}$ be a control steering \eqref{eqn:CS_deg_1} from $\zero$ to $\vv$ (it exists, since $\vv$ belongs to the reachable set $\RR(t_0',\tau)$), and let $\du_2\in\LL{\tau}{t_1'}$ be a control steering \eqref{eqn:CS_deg_1} from $\vv$ to $\zero$. Consider the concatenated control $\du=\du_1|_{[t_0',\tau)}\circ\du_2|_{[\tau,t_1']}$. The related trajectory $\qq{1}(t,\du)$ satisfies $\qq{1}(\tau,\du)=\vv$ and $\qq{1}(t_1',\du)=\zero$. By the results of Lemma~\ref{lem:fundamental} we have
\begin{align*}
    \cost{t_0',\tau}{\du}-\costmin{t_0',\tau}{\vv=\qq{1}(\tau,\du)}\leq &\cost{t_0',t_1'}{\du}-\costmin{t_0',t_1'}{\zero=\qq{1}(t_1',\du)}\overset{\text{Prop.~\ref{fact:Q_zero_is_zero}}}=\\
    &\cost{t_0',t_1'}{\du}-0\ ,
\end{align*}
hence $\costmin{t_0',\tau}{\vv}$ is bounded from below. 
\end{proof}

\paragraph{Step 3 -- controllability to zero}
Now we prove that, after a suitable division of a SR trajectory (not necessarily an abnormal geodesic), the characteristic control system \eqref{eqn:CS_deg_1} is controllable to zero from every point of its reachable set. We start with a few simple properties of the reachable sets.  

 \begin{proposition}
\label{prop:reach_set}
Consider a SR trajectory $q(t)$, with $t\in[t_0,t_1]$, corresponding to a control $u\in\LL{t_0}{t_1}$. The reachable sets $\RR(t_0,\tau)$ of the related characteristic control system have the following properties:
\begin{enumerate}[(i)]
    \item for each  $\tau\in (t_0,t_1]$, the set $\RR(t_0,\tau)$ is a linear subspace of $\R^{n+1}$;
    \item\label{point:2} if $\tau<\tau'$ then $\RR(t_0,\tau)\subseteq\RR(t_0,\tau')$;
    \item there exists $t_1'\in(t_0,t_1)$ such that $\RR(t_0,\tau)=\RR(t_0,t_1)$ for every $\tau\in(t_1',t_1]$.
\end{enumerate}
\end{proposition} 
\begin{proof}
The first point is obvious as the system is linear. So is the second, since each control $u$ defined on $[t_0,\tau]$ can be extended to a control on $[t_0,\tau']$ by setting $u(t)=0$ on $(\tau,\tau']$.\smallskip

The third point follows from \eqref{point:2} and the fact that the monotonous integer-valued function $\tau\mapsto \dim \RR(t_0,\tau)$ is left semi-continuous. Indeed, assume that for some $\tau_0\in [t_0,t_1]$ we have $\dim \RR(t_0,\tau_0)>\dim \RR(t_0,\tau)$ for all $\tau<\tau_0$. If this is the case then there exist a covector $\covector\in(\R^{n+1})^\ast$ annihilating $\RR(t_0,\tau)$ for every $\tau< t_0$, but not-vanishing on $\RR(t_0,\tau_0)$. If so, choose a control $\du_0$ such that 
$$\<\covector, \qq{1}(\tau_0,\du_0)>\neq 0\ .$$
The map $\tau\mapsto \<\covector, \qq{1}(\tau,\du_0)>$ is, however, defined by an integral of a MB function over $[t_0,\tau]$ and thus is continuous with respect to $
\tau$, so necessarily $\<\covector, \qq{1}(\tau,\du_0)>\neq 0$  for some $\tau<\tau_0$. On the other hand, $\qq{1}(\tau,\du_0)\in \RR(t_0,\tau)\subset \ker \covector$. The contradiction ends the proof. 
\end{proof}

\begin{center}
\includegraphics[scale=0.6]{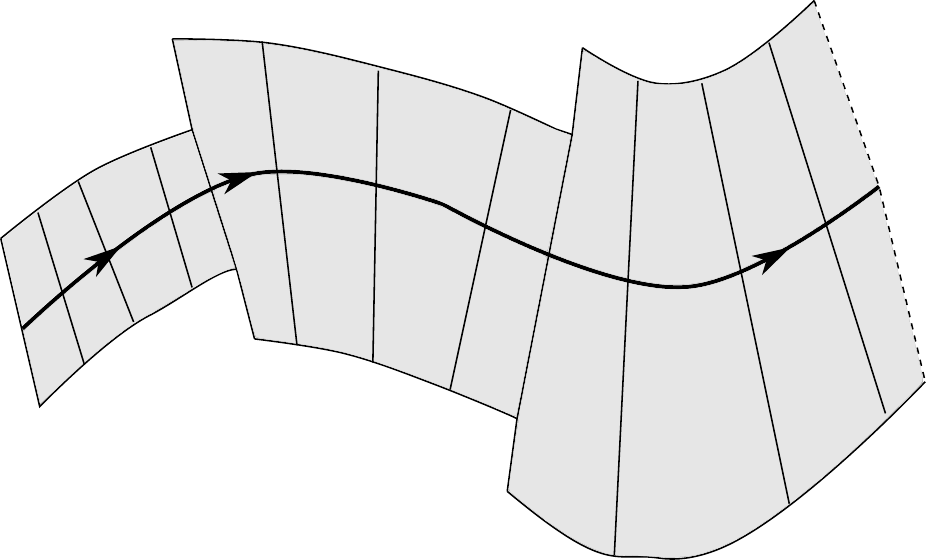}
\put(-300,-25){The dimension of the reachable set $\RR(t_0,t)$ grows along the trajectory $q(t)$.}
\put(-65,90){$q(t)$}
\put(-65,165){$\RR(t_0,t)$}
\end{center}

It turns out that after cutting the trajectory to a finite number of pieces all elements of the reachable sets are controllable to zero.

\begin{lemma}\label{lem:control_to_zero} Consider a SR trajectory $q(t)$, with $t\in[t_0,t_1]$, corresponding to a control $u\in\LL{t_0}{t_1}$. There exist at most $n+1$-points $t_0=\tau_0<\tau_1<\hdots< \tau_s=t_1$ such that for each $\tau\in [\tau_i,\tau_{i+1})$ all points of the reachable set $\RR(\tau_i,\tau)$ of the related characteristic control system are controllable to zero from time $\tau$.
\end{lemma}
\begin{proof}
Denote by $(\overline{\Lambda_1})$ the control system \eqref{eqn:CS_deg_1} with reversed time (i.e we control from $t_1$ to $t_0$), and let $\bm{\mathcal{R}}_{(\overline{\Lambda_1})}(t_1,\tau)$ denote its time $\tau$-reachable set. Obviously, a vector $\vv\in\R^{n+1}$ is controllable to zero from time $\tau\in[t_0,t_1]$ if and only if $\vv\in \bm{\mathcal{R}}_{(\overline{\Lambda_1})}(t_1,\tau)$. Observe that $\bm{\mathcal{R}}_{(\overline{\Lambda_1})}(t_1,t_0)=\RR(t_0,t_1)$ as both linear spaces are defined by the same integrals differing only by the direction of integration. System $(\overline{\Lambda_1})$ has properties analogous to those of \eqref{eqn:CS_deg_1} stated in Prop.~\ref{prop:reach_set}. Therefore there exists a maximal $\tau_1>t_0$ such that for every $\tau\in[t_0,\tau_1)$
$$\bm{\mathcal{R}}_{(\overline{\Lambda_1})}(t_1,\tau)=\bm{\mathcal{R}}_{(\overline{\Lambda_1})}(t_1,t_0)=\RR(t_0,t_1)\overset{\text{Prop.~\ref{prop:reach_set} \eqref{point:2}}}\supset \RR(t_0,\tau)\ ,$$
Proving the assertion for $i=0$. Now repeat the argument to the system on the new interval $[\tau_1,t_1]$. Note that by choosing $\tau_1$ maximal we have $\dim \RR(\tau_1,t_1)=\dim \bm{\mathcal{R}}_{(\overline{\Lambda_1})}(t_1,\tau_1)<\dim \bm{\mathcal{R}}_{(\overline{\Lambda_1})}(t_1,t_0)=\dim \RR(t_0,t_1)$, so the procedure will end in at most $n$ steps when the dimension of the reachable set falls to zero.  
\end{proof} 

\paragraph{Conclusion -- a suitable division of the SR geodesic}

Now we can combine the two divisions: the one described in Corollary~\ref{cor:AS_index_zero_1} which guarantees finiteness of the minimal 2-jet at zero, and the one from Lemma~\ref{lem:control_to_zero} which guarantees controllability to zero of any point of the reachable set, to get: 

   \begin{lemma}
        \label{lem:final_division} Assume that a SR trajectory $q(t)$, with $t\in[t_0,t_1]$, corresponding to a control $u\in\LL{t_0}{t_1}$, is a minimizing abnormal SR geodesic of co-rank $r$. Then there exists a non-zero covector $\psi_0\in(\R^n)^\ast$, and at most $r(n-r-k+1)+1$ points $\tau_0=t_0 < \tau_1<\hdots<\tau_N=t_1$ such that on each sub-interval $[t_0',t_1']\subset(\tau_i,\tau_{i+1})$ the related value function $\costmin{t_0',\tau}{\vv}$ is finite for every $\tau\in [t_0',t_1']$ and every $\vv\in\RR(t_0',\tau)$.  
        \end{lemma}
\begin{proof}
We first choose $\psi_0$ as in Lemma~\ref{lem:AS_index_zero_1} and, according to Corollary~\ref{cor:AS_index_zero_1}, divide $[t_0,t_1]$ into at most $r$ pieces of the form $(\tau_j,\tau_{j+1}]$ such that the  $\costmin{t_0',\tau_{j+1}}{\zero}=0$ for any $[t_0',\tau_{j+1}]\subset(\tau_j,\tau_{j+1}]$. Now by Lemma~\ref{lem:control_to_zero}  we divide $[t_0',\tau_{j+1}]$ into at most $n$ pieces to guarantee that on each of these all elements of the reachable set are controllable to zero. By Corollary~\ref{cor:fundamental} the minimal 2-jet
$\costmin{t_0',\tau}{\vv}$ is finite for all elements $\vv$ of the reachable set. Actually, it follows from the proof of Lemma~\ref{lem:control_to_zero} that those new division points are the points at which the dimension of the reachable set $\bm{\mathcal{R}}_{(\overline{\Lambda_1})}(\tau_{j+1},\tau)$ (for system \eqref{eqn:CS_deg_1} with reversed time) grows. Thus, first of all, these new division points do not depend on the choice of $t_0'$, hence in fact we have a division of the whole $[\tau_j,\tau_{j+1}]$. Secondly, the number of pieces of the new division is actually not greater than $n-r-k+1$, as the dimension of $\bm{\mathcal{R}}_{(\overline{\Lambda_1})}(\tau_{j+1},\tau)$ is at most $n+1-r$ (by the assumption that $q(t)$ is of co-rank $r$) and at least $k$ (as \eqref{eqn:CS_deg_1} is controlled by $k$ independent vector fields). This ends the proof. 
\end{proof}

\paragraph{Step 4 -- the shape of the minimal 2-jet}
\label{ssec:shape_min_2_jet}

Now let us address the situation when the minimal 2-jet $\costmin{t_0,\tau}{\vv}$ is finite for all possible choices of $\tau\in[t_0,t_1]$ and $\vv\in\RR(t_0,\tau)$. By the results of Lemma~\ref{lem:final_division} such a situation can be achieved by dividing a minimizing abnormal geodesics into a finite number of pieces.  Our goal is to derive the properties of $\costmin{t_0,\tau}{\vv}$ in such a situation. 

        \begin{theorem}
        \label{thm:properties_of_Q} 
        Consider the characteristic OCP of the trajectory $q(t)$ corresponding to the control $u\in \LL{t_0}{t_1}$. Assume that for some covector $\psi_0\in(\R^n)^\ast$ the map $\costmin{t_0,\tau}{\vv}$ is finite for all possible choices of $\tau\in[t_0,t_1]$ and $\vv\in\RR(t_0,\tau)$. Then 
        \begin{itemize}
            \item $\costmin{t_0,\tau}{\vv}$ is quadratic with respect to $\vv$.
            \item for a fixed $\vv$, function $\tau\mapsto \costmin{t_0,\tau}{\vv}$ is non-increasing. 
        \end{itemize}
        \end{theorem}
\begin{proof}
To prove the first part of the assertion, fix $\tau\in[t_0,t_1]$ choose $\eps >0$, and choose any basis $\{\vv_1,\hdots,\vv_s\}$ of $\RR(t_0,\tau)$. For $i=1,2,\hdots,s$ denote $\vv_{s+i}:=-\vv_i$  and let  $P$ be the convex polygon with vertices $\pm\vv_i$. Since $\costmin{t_0,\tau}{\cdot}$ is defined as an infimum, for every $j=1,2,\hdots,2s$ we may find a control $\du_j\in\LL{t_0}{\tau}$ such that $\bbb{1}(\tau,\du_j)=\vv_j$ and $\cost{t_0,\tau}{\du_j}\leq \costmin{t_0,\tau}{\vv_j}+\eps$. Since for every $\dv\in\U{0}(t_0,\tau)$ we have
$$\cost{t_0,\tau}{\du_j+\dv}=\cost{t_0,\tau}{\du_j}+2\blincost{t_0,\tau}{\du_j}{\dv}+\cost{t_0,\tau}{\dv}\geq \costmin{t_0,\tau}{\vv_j=\bbb{1}(\tau,\du_j+\dv)}\ ,$$
the  condition $\cost{t_0,\tau}{\du_j}\leq \costmin{t_0,\tau}{\vv_j}+\eps$ is equivalent to the following one:
$$2\blincost{t_0,\tau}{\du_j}{\dv}+\cost{t_0,\tau}{\dv}\geq -\eps\qquad \text{for every $\dv\in\U{0}(t_0,\tau)$.}$$

Now for every convex combination $\sum_{j=1}^{2s}\lambda_j \cdot \du_j$ consider $F(\sum _j\lambda_j\cdot \vv_j):=\cost{t_0,\tau}{\sum_j\lambda_j \cdot \du_j}$. Obviously, as $\bbb{1}(\tau, \sum_j\lambda_j \cdot \du_j)=\sum_j\lambda_j\cdot \bbb{1}(\tau,\du_j)=\sum_j\lambda_j\cdot\vv_j$, and $\cost{t_0,\tau}{\cdot}$ is a quadratic function, then $F$ is a quadratic function on the polygon $P$. We further claim that $\costmin{t_0,\tau}{\vv}\leq F(\vv)\leq \costmin{r_0,\tau}{\vv}+\eps$. The first inequality is obvious as $\costmin{t_0,\tau}{\vv}$ is the infimum value of the functional $\cost{t_0,\tau}{\cdot}$ and $F(\vv)=\cost{t_0,\tau}{\sum_j\lambda_j \cdot \du_j}$ a particular one. The second inequality follows easily by an addition of analogous inequalities for particular $\du_j$'s. Namely for every $\dv\in\U{0}(t_0,\tau)$ we have
\begin{align*}
 2\blincost{t_0,\tau}{\sum_j\lambda_j\cdot \du_j}{\dv}+\cost{t_0,\tau}{\dv}= 2\blincost{t_0,\tau}{\sum_j\lambda_j\cdot \du_j}{\dv}+(\sum_j\lambda_j)\cdot\cost{t_0,\tau}{\dv}=\\
    \sum_{j=1}^{2s}\lambda_j\cdot \left(2\blincost{t_0,\tau}{\du_j}{\dv}+\cost{t_0,\tau}{\dv}\right)\geq \sum_{j=1}^{2s}\lambda_j\cdot(-\eps)=-\eps\ .
\end{align*}

At the end of the day we showed that the function $\costmin{t_0,\tau}{\vv}$ can be $\eps$-uniformly approximated by a quadratic function $F(\vv)$ on the polygon $P\ni\vv$. Note that a uniform limit of a sequence of quadratic maps is also a quadratic map, thus $\costmin{t_0,\tau}{\cdot}$ is quadratic on $P$. Since $P$ was arbitrary the assertion follows. \medskip

The second part of the assertion is straightforward. As extending a given control $\du\in\LL{t_0}{\tau}$ to a control $\du'\in\LL{t_0}{\tau'}$ by a concatenation with the zero control, i.e. $\du'=\du|_{[t_0,\tau]}\circ0|_{[\tau,\tau']}$, gives 
$\cost{t_0,\tau}{\du}=\cost{t_0,\tau'}{\du'}$, we have $\costmin{t_0,\tau}{\vv}\geq \costmin{t_0.\tau'}{\vv}$ for every $\tau'>\tau$ and $\vv\in\RR(t_0,\tau)$. 
\end{proof}

\paragraph{End of the proof}
Now we have all the pieces ready to conclude the proof of Theorem~\ref{thm:optimal_deg_2}. Namely, Lemma~\ref{lem:final_division} gives us a division of an abnormal SR geodesic $q(t)$ into a finite number of pieces, such that $\costmin{t_0',\tau}{\vv}$ is finite for every possible $\vv\in \RR(t_0',\tau)$, where $[t_0',t_1']\ni \tau$ lies inside a single piece. By the definition of the minimal 2-jet as an infimum, we know that 
$$\cost{t_0',\tau}{\du}-\costmin{t_0',\tau}{\qq{1}(\tau,\du)}\geq 0\ .$$
Further the map
\begin{align*}
    \tau\longmapsto\cost{t_0',\tau}{\du}-\costmin{t_0',\tau}{\qq{1}(\tau,\du)}=\<\qq{2}(t,\du),\psi_0>-\costmin{t_0',\tau}{\qq{1}(\tau,\du)}
\end{align*}
is non-increasing by Lemma~\ref{lem:fundamental}. What is more $\costmin{t_0',\tau}{\vv}$ is quadratic in $\vv$ by the results of Theorem~\ref{thm:properties_of_Q}. Finally, as the map $\tau\mapsto \costmin{t_0',\tau}{\vv}$ is monotone, the coefficients of this quadratic form (say, in the standard basis of $\R^{n+1}$) are differences of a finite number of monotone function. Consequently, they are of bounded variation.  
\smallskip

Summing up all the above observations, we have proven a version of Theorem~\ref{thm:optimal_deg_2}  with the requested quadratic map equal to $-\costmin{t_0',\tau}{\vv}$. The only problem is that the data is expressed in terms of adapted coordinates $(\qq{1}(t,\du),\qq{2}(t,\du))$, instead of jets $(\bbb{1}(t,\du), \bbb{2}(t,\du))$ as in the original formulation. 
To end the proof we have just to invert the transformation of adapted coordinates \eqref{eqn:def_q12_du}. Note that, by the results of Proposition~\ref{prop:pontryagin_in_adapted}, $\cost{t_0',\tau}{\qq{1}(\tau,\du)}=\<\q{2}(t,\du),\psi_0>$ equals to $\<\bb{2}(t,\du),\varphi(t)>$ up to a term quadratic in $\bb{1}(t,\du)$; and that the passage between $\bbb{1}(t,\du)$ and $\qq{1}(t,\du)$ is linear. Therefore the quadratic character of the minimal 2-jet will be preserved under this transformation. So, still by the results of Proposition~\ref{prop:pontryagin_in_adapted}, would be the bounded variation regularity of this map. \qed

\section{Supplementary results -- on solutions of the characteristic OCP}
\label{sec:suplementary}

In Ssec.~\ref{ssec:shape_min_2_jet} we were able to describe the shape of the minimal 2-jet $\costmin{t_0,t}{\vv}$, i.e the infimum of the set of values of $\cost{t_0,t}{\du}$ under the condition that $\qq{1}(t,\du)=\vv$. This, in turn, allowed us to prove Thm~\ref{thm:optimal_deg_2} describing sub-Riemannian optimality conditions of degree two. So far, however, we didn't address the question whether these infimum values are realized by some controls, i.e. if the characteristic OCP actually has solutions. In general this will not happen, as may be seen in the examples discussed in Subsection~\ref{ssec:examples}. The reason is that   $\costmin{t_0,t}{\vv}$ is defined as an infimum  taken over elements $\du$ belonging to an affine subspace of finite codimension in $\LL{t_0}{t_1}$. And we lack any sort of universal (weakly) compactness results in such situations. Below we will discuss some conditions that can guarantee the existence of the infimum. 
\smallskip

First of all, in Subsection~\ref{ssec:algebraic}, we study the characteristic OCP from the algebraic point of view, exploiting mostly the fact that the problem is quadratic in controls. It turns out that an important role for the problem in general is played by $\Sol{t_0,t}{(\zero)}$ -- the space of solutions of the characteristic OCP with the end-point $\zero$. We identify that one of the assumptions made in \cite{Monti_third_order_2020} implies that this space has finite codimension in the space of all controls $\LL{t_0}{t}$.  We formalize this as \emph{Monti's condition} and prove, in Lemma~\ref{lem:monti_gives_solutions}, that this finite codimension assumption implies that the characteristic OCP has solutions for an arbitrary end-point.   
\smallskip

If the existence of solutions of the characteristic OCP can be somehow assured, we can apply to them the Pontryagin Maximum Principle to get additional information about the underlying SR trajectory $q(t)$. Precisely this situation is considered in Subsection~\ref{ssec:pmp}. In particular, by combining our main result with the conclusions of the PMP, we prove  Lemma~\ref{lem:monti} stating that if Monti's conditions are assumed,  then the abnormal SR geodesics consist of a finite number of pieces which are either $C^2$ or 2-abnormal.  

\subsection{Algebraic approach to the characteristic OCP, Monti's condition}\label{ssec:algebraic}

\paragraph{Algebraic criteria for the existence of solutions}

We begin our discussion with a simple algebraic criteria for a control to be a solution of the characteristic OCP.

        \begin{proposition}
        \label{fact:sol_unbounded_bilin}
        Let $\du\in\LL{t_0}{t}$ be such that $\qq{1}(t,\du)=\vv$. Then $\du$ is a solution of the characteristic OCP if and only if 
        \begin{enumerate}[(i)]
            \item\label{cond:I} $\blincost{t_0,t}{\du}{\dv}= 0$ for every $\dv\in\U{0}(t_0,t)$ and
            \item\label{cond:II} $\costmin{t_0,t}{\zero}$ is finite (hence equal to zero by Prop.~\ref{fact:Q_zero_is_zero}).
        \end{enumerate}
        \end{proposition}
      
\begin{proof}
First let us prove that conditions \eqref{cond:I}--\eqref{cond:II} imply optimality of the control $\du$. 
Note that controls $\du,\du'\in\LL{t_0}{t}$ are such that $\qq{1}(t,\du)=\qq{1}(t,\du')=\vv$ if and only if $\dv:=\du'-\du$ belongs to $\U{0}(t_0,t)$. Now by our assumptions $\cost{t_0,t}{\dv}\geq\costmin{t_0,t}{\zero}=0$ and $\blincost{t_0,t}{\du}{\dv}= 0$.  Therefore we have
\begin{align*}\cost{t_0,t}{\du'}=\cost{t_0,t}{\du+\dv}=&\cost{t_0,t}{\du}+2\blincost{t_0,t}{\du}{\dv}+\cost{t_0,t}{\dv}=\\
&\cost{t_0,t}{\du}+\cost{t_0,t}{\dv}\geq \cost{t_0,t}{\du}\ ,
\end{align*}
 Since $\du'$ was arbitrary, $\du$ is indeed a solution of the characteristic OCP.
\medskip

Assume now that $\du$ is a solution of the characteristic OCP. This implies, in particular, that $\costmin{t_0,t}{\vv}$ is finite for $\vv=\qq{1}(t,\du)$. The latter condition suffices to prove \eqref{cond:II}.  Indeed, take any $\dv\in\U{0}(t_0,t)$ and any $\lambda\in\R$. Then
\begin{align*}
   \cost{t_0,t}{\du+\lambda\dv}=&\cost{t_0,t}{\du}+2\blincost{t_0,t}{\du}{\lambda\dv}+\cost{t_0,t}{\lambda\dv}=\\
   &\cost{t_0,t}{\du}+2\lambda\cdot\blincost{t_0,t}{\du}{\dv}+\lambda^2\cdot\cost{t_0,t}{\dv}\ . 
\end{align*}
Thus if $\cost{t_0,t}{\dv}$ would be negative, we could make $\cost{t_0,t}{\du+\lambda \dv}$ as low as possible by taking $\lambda$ big enough. Therefore the assumption of finteness of $\costmin{t_0,t}{\vv}$ requires $\cost{t_0,t}{\dv}\geq 0$ for every $\dv\in \U{0}(t_0,t)$. In consequence, $\costmin{t_0,t}{\zero}\geq 0$ and thus $\costmin{t_0,t}{\zero}=0$ by Prop.~\ref{fact:Q_zero_is_zero}.
\smallskip

To prove point \eqref{cond:I} take $\dv\in \U{0}(t_0,t)$. Then for each $\lambda \in\R$ the control $\du+\lambda \dv$ belongs to $\LL{t_0}{t}$, and ends at $\qq{1}(t,\du+\lambda\cdot \dv)=\vv$. Thus the map
$$\R\ni\lambda \longmapsto \cost{t_0,t}{\du+\lambda \dv}=\cost{t_0,t}{\du}+2\lambda\cdot\blincost{t_0,t}{\du}{\dv}+\lambda^2\cdot\cost{t_0,t}{\dv}$$
has a minimum at $\lambda=0$, which implies  $\blincost{t_0,t}{\du}{\dv}= 0$.
\end{proof}

\paragraph{Spaces of solution of the characteristic OCP} 
We conclude that the space of solutions of the characteristic OCP has a simple structure.

\begin{proposition}
The set of solutions of the characteristic OCP 
$$\Sol{t_0,t}{}:=\{\du\in \LL{t_0}{t_1}\that \cost{t_0,t}{\du}=\costmin{t_0,t}{\qq{1}(t,\du)}\}$$
is a vector space. Further the space of solutions satisfying $\qq{1}(t,\du)=\vv$ 
$$\Sol{t_0,t}{(\vv)}:=\{\du\in \LL{t_0}{t_1}\that \cost{t_0,t}{\du}=\costmin{t_0,t}{\vv}, \qq{1}(t,\du)=\vv\}$$
is an affine space modeled on 
$$\Sol{t_0,t}{(\zero)}=\{\du\in \LL{t_0}{t_1}\that \cost{t_0,t}{\du}=\costmin{t_0,t}{\zero}, \qq{1}(t,\du)=\zero\}\ .$$
\end{proposition} 

\begin{proof}
By Prop.~\ref{fact:sol_unbounded_bilin}, each solution $\du\in \LL{t_0}{t}$ is characterized by the condition that the map
$$\blincost{t_0,t}{\du}{\cdot}:\U{0}(t_0,t)\lra \R$$
is null. Note that $\blincost{t_0,t}{\du}{\cdot}$ is linear in $\du$. This ends the first part of the proof.
\smallskip

Let now $\du,\du'$ be two elements of $\Sol{t_0,t}(\vv)$. Then, of course $\dv:=\du'-\du$ satisfies $\qq{1}(t,\dv)=\zero$. Further
\begin{align*}
    0=\costmin{t_0,t}{\vv}-\costmin{t_0,t}{\vv}=\cost{t_0,t}{\du'}-\cost{t_0,t}{\du}=&\cost{t_0,t}{\du+\dv}-\cost{t_0,t}{\du}= \\
    &2\blincost{t_0,t}{\du}{\dv}+\cost{t_0,t}{\dv}\ .
\end{align*}
By the results of Prop.~\ref{fact:sol_unbounded_bilin} the bilinear term vanishes, and so $\cost{t_0,t}{\dv}=0=\costmin{t_0,t}{\zero}$, i.e. $\du'-\du\in\Sol{t_0,t}{(\zero)}$. Finally if put  $\vv=\zero$ in the above calculation, we would get that a difference of two elements of $\Sol{t_0,t}{(\zero)}$ lies in $\Sol{t_0,t}{(\zero)}$. We conclude that the latter is a vector space. This ends the proof.
\end{proof}

The space $\Sol{t_0,\tau}{(\zero)}$ has some interesting properties.

        \begin{proposition}
        \label{prop:sol_zero}
        Assume that the minimal 2-jet $\costmin{t_0,\tau}{\vv}$ is finite for every $\vv\in\RR(t_0,\tau)$. Then for every $\du\in\LL{t_0}{t_1}$, and every $\dv\in \Sol{t_0,\tau}{(\zero)}$ we have
        $$\blincost{t_0,\tau}{\du}{\dv}=0\ .$$  
        \end{proposition}

\begin{proof}
Take any such $\du$ and $\dv$. Since $\cost{t_0,\tau}{\dv}=\costmin{t_0,\tau}{\zero}=0$, for every $\lambda\in\R$ we have
\begin{align*}
    \cost{t_0,\tau}{\du+\lambda\cdot \dv}=\cost{t_0,\tau}{\du}+2\lambda\blincost{t_0,\tau}{\du}{\dv}+\lambda^2\cost{t_0,\tau}{\dv}=\cost{t_0,\tau}{\du}+2\lambda\blincost{t_0,\tau}{\du}{\dv}
\end{align*}
Note that $\qq{1}(\tau,\du+\lambda\dv)=\qq{1}(\tau,\du)$, as $\qq{1}(\tau,\dv)=0$. If $\blincost{t_0,\tau}{\du}{\dv}\neq0$, then by taking $\lambda$ big or small enough we would find a control steering \eqref{eqn:CS_deg_1} from $\zero$ to $\qq{1}(\tau,\du)$ with arbitrarily small cost. This contradicts the finiteness of $\costmin{t_0,\tau}{\qq{1}(\tau,\du)}$. \smallskip
\end{proof}

\paragraph{Monti's condition}
Prop.~\ref{fact:sol_unbounded_bilin} shows that solutions of the characteristic OCP are related with the kernel of the linear maps $\blincost{t_0,\tau}{\du}{\cdot}:\U{0}(t_0,\tau)\ra \R$. On the other hand, the previous Prop.~\ref{prop:sol_zero} shows that the space $\Sol{t_0,\tau}{(\zero)}$ is contained in that kernel. It turns out that if the space $\Sol{t_0,\tau}{(\zero)}$ is big enough it may guarantee the existence of solutions of the characteristic OCP. To formalize this property let us propose the following definition.

        \begin{definition}
        \label{def:monti_cond}
        Let $q(t)$ with $t\in[t_0,t_1]$ be a sub-Riemannian trajectory and  consider the related characteristic OCP. We say that the trajectory satisfies \emph{Monti's condition} at time $\tau\in[t_0,t_1]$ if the space $\Sol{t_0,\tau}{(\zero)}$ is of finite codimension in $\LL{t_0}{\tau}$ (equivalently: of finite codimension in $\U{0}(t_0,\tau)$, as the latter  is a space of finite codimension in $\LL{t_0}{\tau}$ being the kernel of the linear map $\qq{1}:\LL{t_0}{\tau}\ra\R^{n+1}$).
        \end{definition}

We propose the name Monti's condition as this is one of the assumptions used in \cite{Monti_third_order_2020} to derive the third-order optimality conditions.\footnote{Actually, the assumption of \cite[Thm 1.2]{Monti_third_order_2020} translated to the notation of this paper is that the space of elements 
$$\operatorname{dom}(\mathcal{D}^3F):=\{\dv\in \LL{t_0}{\tau}\that \qq{1}(\tau,\dv)=0\ \text{and $\qq{11}(\tau,\du,\dv)\in \image\qq{1}(\tau,\cdot)$ for every $\du\in \LL{t_0}{t}$} \}$$ is of finite co-dimension in $\U{0}(t_0,t)$. Here, $2 \qq{11}(\tau,\du,\dv):=\qq{2}(\tau,\du+\dv)-\qq{2}(\tau,\du)-\qq{2}(\tau,\dv)$ is the bi-linear form related with the quadratic map $\qq{2}(\tau,\cdot)$. 

Note that, in particular, every element $\dv\in\operatorname{dom}(\mathcal{D}^3F)$ satisfies $\qq{11}(\tau,\dv,\dv)=\qq{2}(t,\dv)\in \image \qq{1}(\tau,\cdot)\subset \ker \psi_0$, where $\psi_0$ is as Prop.~\ref{prop:pontryagin_in_adapted}. Consequently, $\cost{\tau}{\dv}=\<\qq{2}(\tau,\dv),\psi_0>=0$ and thus $\cost{\tau}{\dv}=0=\costmin{\tau}{\zero}$, i.e. $\dv\in\Sol{\tau}{(\zero)}$ whenever $\costmin{\tau}{\zero}$ is finite. Summing up, if $\costmin{\tau}{\zero}$ is finite then
$$\operatorname{dom}(\mathcal{D}^3F)\subset\Sol{\tau}{(\zero)}\subset\LL{t_0}{\tau}\ $$
and thus the condition of Definition~\ref{def:monti_cond} is weaker then the assumption made in \cite{Monti_third_order_2020}.} This assumption is, however, a strong one. Further it may not be satisfied, as is the case in the two examples discussed in Subsection~\ref{ssec:examples}. 

        \begin{lemma}
        \label{lem:monti_gives_solutions}
        Let $q(t)$ with $t\in[t_0,t_1]$ be a sub-Riemannian trajectory. Consider the related characteristic OCP and assume that the minimal 2-jet $\costmin{t_0,\tau}{\vv}$ is finite for every $\vv\in\RR(t_0,\tau)$. If the trajectory satisfies Monti's condition at time $\tau$, then the set $\Sol{t_0,t}{(\vv)}$ is non-empty for every $\vv\in\RR(t_0,\tau)$.
        \end{lemma}

\begin{proof}
We will show that Monti's condition actually reduces the characteristic OCP to a finite-dimensional optimization problem. 

Consider a control $\du\in\LL{t_0}{t_1}$ satisfying $\qq{1}(\tau,\du)=\vv$, and let us introduce any splitting $\U{0}{(t,\tau)}=V\oplus \Sol{t_0,\tau}{(\zero)}$, where $V$ is finite-dimensional. The quadratic map $\dv\mapsto \cost{t_0,\tau}{\dv}=\blincost{t_0,\tau}{\dv}{\dv}$ is positively-defined on $V$, and so we may find a $\blincost{t_0,\tau}{\cdot}{\cdot}$-orthonormal basis $\{\dv_\alpha\}_{\alpha\in A}$ of $V$. Denote $\lambda_\alpha:=\blincost{t_0,\tau}{\du}{\dv_\alpha}$ and consider $\du'=\du-\sum_{\alpha\in A}\lambda_\alpha\cdot \dv_\alpha$. By construction $\qq{1}(\tau,\du')=\qq{1}(\tau,\du)=\vv$, and it is straightforward to check that $\blincost{t_0,\tau}{\du'}{\dv_\alpha}=0$ for every $\alpha\in A$. Since, by Prop.~\ref{prop:sol_zero}, also $\blincost{t_0,\tau}{\du'}{\dv}=0$ for every $\dv\in \Sol{t_0,\tau}{(\zero)}$, the linear map $\blincost{t_0,\tau}{\du'}{\cdot}$ vanishes on the whole $\U{0}{(t,\tau)}=V\oplus \Sol{t_0,\tau}{(\zero)}$, and thus $\cost{t_0,\tau}{\du'}=\costmin{t_0,\tau}{\vv}$ by Prop.~\ref{fact:sol_unbounded_bilin}.
\end{proof}

\subsection{Solutions of the characteristic OCP in light of the Pontryagin Maximum Principle}
\label{ssec:pmp}

Once the existence of solutions of the characteristic OCP is established (for example by assuming Monti's condition), we may use the Pontryagin Maximum Principle (PMP, in short) \cite{Pontr_Inn_math_theor_opt_proc_1962} to derive necessary conditions for optimality. They will give us a lot of information on the trajectory. When discussing solutions of the characteristic OCP we will always speak about normal or abnormal \emph{extremals of the characteristic OCP}, not to be confused with the \emph{sub-Riemannian extremal} that constitutes the characteristic OCP itself. 

\paragraph{The PMP for the characteristic OCP}
 For the characteristic OCP the \emph{Pontryagin Hamiltonian} $H:\R\times\R^{n+1}\times(\R^{n+1})^\ast\times\R^k\lra \R$ is defined as \begin{align*}
    H(t,(\q{1},\cc{1}),&(\psi,a),\dv)=H(t,\q{1},(\psi,a),\dv):=\\
    &\<\sum_i\dv_i\cdot \Y{i}{1}(t),\psi>+a\sum_i\dv_i\cdot u_i(t)+p_0\<\sum_i \dv_i\cdot \Y{i}{2}(t,\q{1}),\psi_0>\ .
\end{align*}
Here, $(\psi,a)\in (\R^{n+1})^\ast$ is a covector dual to the pair $(\q{1},\cc{1})\in \R^{n+1}$ and $p_0$ is either -1 (in the normal case) or $0$ (in the abnormal case). Actually, we may assume that  the pair $(\psi,a)$ belongs to the dual of the reachable set $\RR(t_0,t_1)$. Note that since $H$ does not depends on $\cc{1}$, the dual object $a$ will always be constant. The following result is immediate.

    \begin{proposition} 
    \label{prop:basic_pmp}
    Let the control $\du\in \LL{t_0}{t_1}$ together with a covector curve $(\psi(t),a(t))$ be an extremal of the the characteristic OCP related with the SR trajectory $q(t)$, corresponding to the control $u\in\LL{t_0}{t_1}$. Then in the abnormal case ($p_0=0$) both $\psi(t)$ and $a(t)$ are constant. In the normal case ($p_0=-1$) $a(t)$ is constant, while $\psi(t)$ is subject to the evolution equation
    \begin{equation}
        \label{eqn:ev_PMP}
        \<b,\dot\psi(t)>=\<\sum_i \du_i(t)\cdot \Y{i}{2}(t,b),\psi_0>
    \end{equation}
    where $b\in \R^n$ is arbitrary. 
    \smallskip
    
    Further, as the Hamiltonian $ H(t,\q{1},(\psi,a),\dv)$ is a linear function of $\dv$, and we have no bounds on the control, the Maximum Principle gives us the following conditions 
    \begin{equation}
        \label{eqn:max_princip}
        \<\Y{i}{1}(t),\psi(t)>+a\cdot u_i(t)+p_0\< \Y{i}{2}(t,\q{1}),\psi_0>=0
    \end{equation}
    for every $i=1,2,\hdots,k$.
    \end{proposition}

\paragraph{Existence of extremals implies regularity of the SR trajectory} The existence of certain extremals of the PMP for the characteristic OCP allows to deduce regularity of the underlying SR trajectory $q(t)$.

        \begin{lemma} 
        \label{lem:properties_of_extremals}
        Under the assumptions of Prop.~\ref{prop:basic_pmp}
        \begin{itemize}
            \item if $\du$ is an abnormal extremal of the characteristic OCP with $a\neq 0$, the SR trajectory $q(t)$ has to be  normal  and thus $C^\infty$-smooth;
            \item if $\du$ is a normal extremal   of the characteristic OCP with $a\neq 0$, then  the SR trajectory $q(t)$ is  $C^1$,
            \item if $\du$ is a normal extremal   of the characteristic OCP with $a\neq 0$, and the SR trajectory $q(t)$ additionally satisfies Goh conditions, then $q(t)$ is  $C^2$. 
        \end{itemize}
        \end{lemma}

\begin{proof}
To prove the first part, let $(\psi,a)\in (\R^{n+1})^\ast$ be the abnormal covector curve of the characteristic OCP (it is constant by Prop.~\ref{prop:basic_pmp}). The Maximum Principle \eqref{eqn:max_princip} gives us in this case
$$\<\Y{i}{1}(t),\psi>+a\cdot u_i(t)=0\quad\text{for $i=1,2,\hdots,k$}.$$
These are precisely the equations of the normal SR trajectory. In particular we get $u_i(t)=-\frac 1a \<\Y{i}{1}(t),\psi>$. Now note that on the left-hand side we have something of class $C^r$ (with $r=0_-,1_-,1,2,\hdots$) while, by the results of Lemma~\ref{lem:adapted}, on the right-hand side something of class $C^{r+1}$. Hence $u_i(t)$ must be $C^\infty$-smooth by the bootstrap argument.\medskip

A similar argument works also for a normal extremal of the characteristic OCP. In that case, by \eqref{eqn:max_princip} we have 
$$u_i(t)=-\frac 1a\left(\<\Y{i}{1}(t),\psi(t)>-\< \Y{i}{2}(t,\q{1}(t)),\psi_0>\right)\ .$$
Now note that, if $u(t)$ was of class $L^\infty$, then, by Lemma~\ref{lem:adapted}, $\Y{i}{1}(t)$ and $\Y{i}{2}(t,\cdot)$ are absolutely continuous with respect to $t$. Further the covector curve $\psi(t)$, as well as the trajectory $\q{1}(t,\du)$ are Lipschitz. Since  $\Y{i}{2}(t,\q{1})$ depends linearly on $\q{1}$, we conclude that $u_i(t)$ is absolutely continuous and, consequently, $q(t)$ is $C^1$. This ends the proof of the second part. \medskip

For the last part we differentiate \eqref{eqn:max_princip} (which we now know is possible almost everywhere) to get
\begin{align*}
    a\dot{u_i}(t)=&-\<\pa_t\Y{i}{1}(t),\psi(t)>-\<\Y{i}{1}(t),\dot\psi(t)>+\< \pa_t\Y{i}{2}(t,\q{1}(t)),\psi_0>+\< \Y{i}{2}(t,\dot{\q{1}}(t)),\psi_0>\overset{\eqref{eqn:ev_PMP},\eqref{eqn:CS_deg_1}}=\\
    &-\<\pa_t\Y{i}{1}(t),\psi(t)>+\< \pa_t\Y{i}{2}(t,\q{1}(t)),\psi_0>\\
    &-\<\sum_j\du_j(t)\cdot \Y{j}{2}(t,\Y{i}{1}(t)),\psi_0>+\< \Y{i}{2}(t,\sum_j\du_j(t)\cdot \Y{j}{1}(t)),\psi_0>=\\
    &-\<\pa_t\Y{i}{1}(t),\psi(t)>+\< \pa_t\Y{i}{2}(t,\q{1}(t)),\psi_0>\\
    &\quad \sum_j \du_j \<\Y{i}{2}(t,\Y{j}{1}(t))-\Y{j}{2}(t,\Y{i}{1}(t)),\psi_0>\overset{\text{Lem.~\ref{lem:adapted}}}=\\
    &-\<\pa_t\Y{i}{1}(t),\psi(t)>+\< \pa_t\Y{i}{2}(t,\q{1}(t)),\psi_0>+\sum_j\du_j \<\lin{1}(t)[X_i,X_j],\psi_0>\ ,
\end{align*}
 By Lemma~\ref{lem:adapted} the $t$-dependence of $\Y{i}{1}(t)$ and $\Y{i}{2}(t,\cdot)$ is $C^1$ (as $u_i(t)$'s are now absolutely continuous), and the term involving the Lie bracket vanishes because of the Goh condition. By repeating our reasoning from the previous point,  $\dot{u_i}(t)$
is absolutely continuous and hence $q(t)$ is actually $C^2$. 
\end{proof}

With a little more effort, combining the information provided by the PMP with our previous results about optimality conditions of degree two, we can prove that actually  solutions of the characteristic OCP are precisely the normal extremals of the PMP. Further, the existence of particular solutions of the characteristic OCP guarantees the $C^2$-regularity of the underlining SR trajectory.

        \begin{lemma}
        \label{lem:properties_of_extremals_2}
        Let $q(t)$ with $t\in [t_0,t_1]$ be a strictly abnormal SR extremal. 
        \begin{enumerate}[(i)]
            \item\label{point:1} Assume that the minimal 2-jet $\costmin{t_0,t}{\zero}$ is finite. Then the control $\du\in \LL{t_0}{t}$ is a solution of the characteristic OCP if and only if it is a normal extremal of the characteristic OCP. 
            \item \label{point:11} Assume that the assertion of Thm~\ref{thm:properties_of_Q} holds. In particular, let $\costmin{t_0,t}{\vv}=\FF(t)[\vv,\vv]$ for a quadratic map $\FF(t):\R^{n+1}\times\R^{n+1}\ra \R$.
            Then if the control $\du\in \LL{t_0}{t}$ is such a solution of the characteristic OCP that  $\FF(t)[(0,1),\qq{1}(t,\du)]\neq 0$, the SR trajectory $q(t)$ is of class $C^2$.
        \end{enumerate}
        \end{lemma}

\begin{proof}
If $\du$ is a solution of the characteristic OCP it must be either a normal, or an abnormal extremal of the PMP. The abnormal case can be easily excluded. Indeed, let $\du$ be an abnormal extremal corresponding to a covector $(\psi,a)\in(\R^{n+1})^\ast$. By the results of previous Lemma~\ref{lem:properties_of_extremals}, if $a\neq 0$, then $q(t)$ must be a normal SR extremal, which is impossible since it is assumed to be strictly abnormal.\smallskip

In the case $a=0$, by \eqref{eqn:max_princip} the covector $\psi$ annihilates all control fields $\Y{i}{1}(t)$ and, in consequence, the whole reachable set of \eqref{eqn:CS_deg_1}, so $\psi=0$ and $a=0$ which is impossible, as the covector defined by the PMP has to be nowhere-vanishing. \medskip

Now we can concentrate on the case where $\du$ is a normal extremal of the characteristic OCP. Let $(\psi(t),a)\in(\R^{n+1})^\ast$ be the related covector curve. For any other control  $\dv\in\LL{t_0}{t_1}$ we would like to calculate the bi-linear form $\blincost{t_0,t}{\du}{\dv}$. Note that
\begin{align*}
\<\sum_i\dv_i(t)\cdot \Y{i}{2}(t,\q{1}(t,\du)),\psi_0>\overset{\eqref{eqn:max_princip}}=&\<\sum_i\dv_i(t)\cdot \Y{i}{1}(t),\psi(t)>+a\cdot \sum_i\dv_i(t)\cdot u_i(t)\overset{\eqref{eqn:CS_deg_1}}=\\
&\<\dotq{1}(t,\dv),\psi(t)>+a\cdot\dot{\cc{1}}(t,\dv)\ .\intertext{On the other hand,}
\<\sum_i\du_i(t)\cdot \Y{i}{2}(t,\q{1}(t,\dv)),\psi_0>\overset{\eqref{eqn:ev_PMP}}=&\<\q{1}(t,\dv),\dot{\psi}(t)>
\end{align*}
By adding the two above equalities and integrating over $[t_0,t]$ we get
\begin{equation}
\label{eqn:B_from_PMP}
\blincost{t_0,t}{\du}{\dv}=\<\q{1}(t,\dv),\psi(t)>+a\cdot \cc{1}(t,\dv)\ .
\end{equation} 
Our first conclusion is that whenever $\dv\in\U{0}(t_0,t)$ then $\q{1}(t,\dv)=0$ and $\cc{1}(t,\dv)=0$, and thus $\blincost{t_0,t}{\du}{\dv}=0$. By Prop.~\ref{fact:sol_unbounded_bilin} this (together with the finiteness of $\costmin{t_0,t}{\zero}$) implies that $\du$ is optimal, proving point \eqref{point:1}. \medskip

To prove \eqref{point:11}, take any $\dv\in\LL{t_0}{t_1}$,  $\eps\in\R$, and denote $\qq{1}(t,\dv)=\vv=(v,c)\in\R^n\times\R$. We have
\begin{align*}\cost{t_0,t}{\du+\eps\cdot\dv}=&\cost{t_0,t}{\du}+\eps\blincost{t_0,t}{\du}{\dv}+\eps^2\cost{t_0,t}{\dv}=\\
&\costmin{t_0,t}{\qq{1}(t,\du)}+\eps\<\vv,(\psi(t),a)>+\eps^2\cost{t_0,t}{\dv}\geq\\
&\costmin{t_0,t}{\qq{1}(t,\du)}+\eps\<\vv,(\psi(t),a)>+\eps^2\costmin{t_0,t}{\vv} \ ,\end{align*}
and thus 
\begin{equation}
    \label{eqn:ineq_Q}
\costmin{t_0,t}{\qq{1}(t,\du)+\eps\vv}\geq \costmin{t_0,t}{\qq{1}(t,\du)}+\eps\<\vv,(\psi(t),a)>+\eps^2\costmin{t_0,t}{\vv}\ .
\end{equation}
On the other hand, by Thm~\ref{thm:properties_of_Q} we know that the minimal 2-jet $\costmin{t_0,t}{\cdot}$ is a quadratic function and hence
\begin{align*}\costmin{t_0,t}{\qq{1}(t,\du)+\eps\vv}=&\FF[\qq{1}(t,\du)+\eps\vv,\qq{1}(t,\du)+\eps\vv]=\\
&\FF[\qq{1}(t,\du),\qq{1}(t,\du)]+2\eps \FF[\vv,\qq{1}(t,\du)]+\eps^2\FF[\vv,\vv]=\\
&\costmin{t_0,t}{\qq{1}(t,\du)}+2\eps \FF[\vv,\qq{1}(t,\du)]+\eps^2\costmin{t_0,t}{\vv}\overset{\eqref{eqn:ineq_Q}}\geq\\
&\costmin{t_0,t}{\qq{1}(t,\du)}+\eps\<\vv,(\psi(t),a)>+\eps^2\costmin{t_0,t}{\vv},
\intertext{implying}
&2\eps\FF[\vv,\qq{1}(t,\du)]\geq \eps\<\vv,(\psi(t),a)>\ .
\end{align*}
As $\eps$ and $\vv=(v,c)$ were arbitrary, the above inequality may hold if and only if 
\begin{align*}\<v,\psi(t)>+c\cdot a=&\<\vv,(\psi(t),a)>=2 \FF[\vv,\qq{1}(t,\du)]\ .
\end{align*}
In particular, taking $\vv=(0,1)$ we get $a=2\FF(t)[(0,1),\qq{1}(t,\du)]$. By Lemma~\ref{lem:properties_of_extremals} if the latter number is different than zero and the Goh conditions are satisfied, the trajectory $q(t)$ is of class $C^2$. Note however, that if the assertion of Theorem~\ref{thm:properties_of_Q} holds, then also the assertion of Theorem~\ref{thm:optimal_deg_2}  is true (with a single piece). To show  it suffices to repeat the reasoning from the last paragraph of Subsection~\ref{ssec:shape_min_2_jet}. But then Goh conditions are true by the results of Subsection~\ref{ssec:discussion}. This ends the proof.
\end{proof}

\paragraph{Regularity under Monit's condition} We end our consideration by stating the following consequences of Monti's condition.
\begin{lemma}
\label{lem:monti}
Let $q(t)$, with $t\in [t_0,t_1]$, be a minimizing strictly abnormal SR geodesic. Assume additionally that it satisfies Monti's condition at time $t_1$ (in the sense of Definition~\ref{def:monti_cond}).

Consider the division of $q(t)$ into a finite number of pieces provided by Thm~\ref{thm:equations_of_motion}. Then on each of these pieces $q(t)$ is either a 2-abnormal extremal in the sense of Definition~\ref{def:2_normal_abnormal}, or of class $C^2$. 
\end{lemma}
\begin{proof}
Consider the division of  $q(t)$ into pieces $[\tau_i,\tau_{i+1}]$ provided by Theorem~\ref{thm:equations_of_motion}. Take any $[t_0',t_1']\subset(\tau_i,\tau_{i+1})$. Let $\PP(t)$ be the related quadratic map and  $a_2(t)\in \R$ its $\R\times\R$-component. Now we have two possibilities, either the piece of $q(t)$, with $t\in[t_0,t_1']$, is a 2-abnormal extremal, or $a_2(t)\neq 0$ for every $t$ (the 2-normal case). Let us assume that we have the latter situation.  
\smallskip

Observe that if the Monti's condition holds at $t_1$ then also it holds at $t_1'$ on the shorter trajectory as all we have to do is to take the finite-codimensional extension $\Sol{t_0,t_1}(\zero)\subset\LL{t_0}{t_1}$ and intersect it with $\LL{t_0'}{t_1'}$.
In particular, by Lemma~\ref{lem:monti_gives_solutions}, the characteristic OCP on $[t_0',t_1']$ has solutions for every possible end-point condition. Let $\du\in\LL{t_0'}{t_1'}$ be such a solution with $\qq{1}(t_1',\du)=(0,1)$. Recall that by the argument form the last paragraph of Subsection~\ref{ssec:shape_min_2_jet} the quadratic maps $\FF(t)$ and $\PP(t)$ have the same $\R\times\R$-component, hence  $$\FF(t_1')[(0,1),\qq{1}(t_1',\du)]=\FF(t_1')[(0,1),(0,1)]=\PP[(0,1),(0,1)]=a_2(t)\neq 0\ .$$
Now Lemma~\ref{lem:properties_of_extremals_2} point \eqref{point:11} guarantees that $q(t)$ is of class $C^2$ on $[t_0',t_1']$. This ends the proof. 
\end{proof}


\section{Further study}\label{sec:conclusion}

The presented material provokes a few interesting questions:
\begin{itemize}
    \item   Is it possible to improve the time-regularity of the minimal 2-jet $\costmin{t_0,\tau}{\vv}$ in Thm.~\ref{thm:properties_of_Q}? Answering this question would be an important step in solving the regularity problem of SR geodesics. Namely, if the $\tau$-regularity of $\costmin{t_0,\tau}{\vv}$ could be improved to, say being of class $C^k$, then in the 2-normal case equation \eqref{eqn:xi_u_1} would guarantee that the control $u(t)$ is also of class $C^k$, and thus the trajectory $q(t)$ be of class $C^{k+1}$ on each piece from a finite collection.
    \smallskip
    
    We \textbf{hypothesize} that it is possible to prove that if the control $u(t)$ is of class $C^r$ (where $r=0_-,1_-,1,2,\hdots$) then the dependence $\tau\mapsto \costmin{t_0,\tau}{\vv}$ is of class $C^{r+1}$. This, by the standard bootstrap argument, will imply that every 2-normal trajectory $q(t)$ is piece-wise smooth (with a finite number of pieces), thus smooth by the results of \cite{no_corners_2016}.
   
    \item For 2-abnormal trajectories, as equation \eqref{eqn:xi_u_1} does not determines the control $u(t)$, the question of regularity remains open regardless of the regularity properties of the minimal 2-jet $\costmin{t_0,\tau}{\vv}$. It seems that for such curves, one needs to understand the geometry of third-order (and higher) order expansion of the extended end-point map. Some research in this direction was initiated in \cite{Monti_third_order_2020}, yet a fully satisfactory theory is yet to be developed.
    
	\item Sub-Riemannian normal extremals are known to be locally minimizing. It is natural to ask whether the same is true for 2-normal extremals. We \textbf{hypothesize} that the answer to this question is positive, and could be provided by extending the methods from \cite{MJ_WR_nsre}.

	\item Our proof is based on Agrachev-Sarychev Index Lemma (Lem.~\ref{lem:AS_index}). It seems to us that is is possible to (slightly) modify the proof of this result to guarantee the following result.\smallskip

    \textbf{Hypothesis}.	
	Let $q(t)$, with $t\in[t_0,t_1]$ be a minimizing strictly abnormal SR geodesic of co-rank $r$. Then there exists a covector $\varphi_0:\T_{q(t_1)}M\ra \R$, and a symmetric 2-form $\PP:\T_{q(t_1)}(M\times\R)\times\T_{q(t_1)}(M\times\R)\ra \R$  such that:
	\begin{itemize}
	    \item $\<\bbb{1}(t_1,\du),\varphi_0>=0$ for every $\du\in \LL{t_0}{t_1}$
	    \item the negative index of the 2-form $$\LL{t_0}{t_1}\ni\du\longmapsto \<\bbb{2}(t_1,\du),\varphi_0>-\PP[\bbb{1}(t_1,\du),\bbb{1}(t_1,\du)]$$
	    is at most $r$. 
	\end{itemize}
	In the standard version the bi-linear term $\PP$ is not present, as the 2-form is restricted to a subspace $\ker \Der_u\END{t_0,t_1}{q_0}\subset \LL{t_0}{t_1}$. 
	\smallskip
	
	It should be also possible to prove that necessarily $\PP[(0,1),(0,1)]\geq 0$, and so $a_2(t)$ described by Thm~\ref{thm:optimal_deg_2} is always non-negative. 
	\smallskip
	
	Such an improvement on Agrachev-Sarychev Lemma will allow to simplify the cutting procedure that we used in order to guarantee the finiteness of all minimal 2-jets $\costmin{t_0',\tau}{\vv}$, and thus reduce the estimate on the maximal number of divisions in Thm~\ref{thm:optimal_deg_2}. 
	
\end{itemize}
We plan to address these questions in a future series of publications. 

\section*{Acknowledgements}
I would like to thank Witold Respondek for bringing my attention to the topic of sub-Riemannian geodesics, our initial work on this topic during my stay in Rouen in 2016, and his vast knowledge of control theory which he eagerly shared with me.  
My special thanks goes to Piotr Mormul for his words of encouragement and support, often expressed in his haiku-like emails.

\bibliographystyle{amsalpha}
\bibliography{bibl}

\end{document}